\documentclass{amsart}
\usepackage{bbm}
\usepackage{bbding}
\usepackage[mathcal]{euscript}
\usepackage{graphicx}
\usepackage{amsthm}
\usepackage{calc}
\usepackage{mathrsfs,dsfont}
\usepackage{CJK,fancyhdr,amscd}
\usepackage{anysize} %control the page
\usepackage{amsmath,amssymb,amsfonts}
\usepackage{epsfig,enumerate}
\usepackage{latexsym}
\usepackage{indentfirst, latexsym}
\usepackage{graphics}
\usepackage[all,poly,knot]{xy}
\usepackage[colorlinks,linkcolor=blue,anchorcolor=blue,citecolor=blue,pagebackref]{hyperref}
\usepackage{geometry}
\numberwithin{equation}{section}
\newtheorem{theorem} {Theorem} [section]
\newtheorem{proposition}[theorem]{Proposition}
\newtheorem{corollary}  [theorem]     {Corollary}
\newtheorem{lemma}  [theorem]     {Lemma}

\newtheorem{remark}  [theorem]     {Remark}

\newtheorem{conjecture}  [theorem]     {Conjecture}

\theoremstyle{definition}
\newtheorem{definition}  [theorem]     {Definition}

\renewcommand*\backref[1]{}
\renewcommand*\backrefalt[4]{ \ifcase #1 \or (cited on page #2) \else (cited on pages #2) \fi}%(no citations)

\begin{document}

\title{Curvature characterization of Hermitian manifolds with Bismut parallel torsion}

\begin{abstract}
In this article, we study Hermitian manifolds whose Bismut connection has parallel torsion, which will be called {\em Bismut torsion parallel manifolds,} or {\em BTP} manifolds for brevity. We obtain a necessary and sufficient condition characterizing BTP manifolds in terms of Bismut curvature tensor alone in Theorem \ref{theorem1.1}. We also present examples and discuss some general properties for BTP manifolds, as well as give a classification result for non-balanced BTP threefolds in Theorem \ref{3DNBBTP}.
\end{abstract}

\author{Quanting Zhao}
\address{Quanting Zhao. School of Mathematics and Statistics, and Hubei Key Laboratory of Mathematical Sciences, Central China Normal University, P.O. Box 71010, Wuhan 430079, P. R. China.} \email{zhaoquanting@126.com;zhaoquanting@mail.ccnu.edu.cn}
\thanks{$\ast$ Zheng is the corresponding author. Zhao is partially supported by NSFC with the grant No.\,12171180, 12371079, 12522107 and by the Fundamental Research Funds for the Central Universities CCNU25JCPT030. Zheng is partially supported by NSFC grants 12141101 and 12471039, by Chongqing Normal University grant 24XLB026, and by the 111 Project D21024.}

\author{Fangyang Zheng$^{\ast}$}
\address{Fangyang Zheng. School of Mathematical Sciences, Chongqing Normal University, Chongqing 401331, China}
\email{20190045@cqnu.edu.cn} \thanks{}

%\date{\today}


\maketitle

\tableofcontents

\section{Introduction and statement of results}

Given a smooth manifold equipped with a connection, the parallelness of torsion often forms strong geometric restrictions. For instance, by the  result of Kamber and Tondeur in \cite{KT}, if a manifold admits a complete connection that is flat and having parallel torsion, then its universal cover must be a Lie group (equipped with a flat left-invariant connection), and the converse is also true. Another example is the classic result of Ambrose and Singer \cite{AS}, which states that if a complete Riemannian manifold admits a metric connection with parallel torsion and curvature, then its universal cover is a homogeneous Riemannian manifold (and vice versa). The complex version of this (proved by Sekigawa in \cite{Sekigawa}) is also true: if a complete Hermitian manifold admits a {\em Hermitian} connection (that is, a connection compatible with both the metric and the almost complex structure) which has parallel torsion and curvature, then the universal cover is a homogeneous Hermitian manifold (and vice versa).

On a given Hermitian manifold, the {\em Bismut connection} (also known as Strominger connection, see \cite{Bismut,Strominger}) is the unique Hermitian connection with totally skew-symmetric torsion. It serves as a bridge between Levi-Civita connection, which is the unique metric connection that is torsion free, and Chern connection, the unique metric connection compatible with the complex structure. When the metric is K\"ahler, these three canonical connections coincide, but when the metric is not K\"ahler, they are mutually distinct, leaving us with three different kinds of geometry.

In this article, we will focus on Bismut connection, and we want to know when will the Bismut connection have parallel torsion. For convenience, we will call the class of Hermitian manifolds whose Bismut connection has parallel torsion simply as {\em Bismut torsion parallel} manifolds, or {\em BTP} manifolds for brevity.

There have been a number  of studies on parallel skew-symmetric torsion in the setting of Riemannian geometry or almost Hermitian geometry, by Friedrich, Agricola, Ferreira, Cleyton, Moroianu, Semmelmann, Schoemann and others, see for example \cite{AF04}, \cite{AF14}, \cite{AFF}, \cite{AFS}, \cite{AFK}, \cite{CMS}, \cite{Sch} and the references therein. Here we carry out the investigation in the slightly more delicate Hermitian situation, where the emphasis is on complex structures, with the hope of understanding the Bismut geometry on complex manifolds.

Let $(M^n,g)$ be a Hermitian manifold. Denote by $\nabla^b$ the Bismut connection and by $T^b$, $R^b$ its torsion and curvature tensor, respectively. Let us introduce the notation
\begin{equation*}
Q_{X\overline{Y}Z\overline{W}} = R^b_{X\overline{Y}Z\overline{W}} - R^b_{Z\,\overline{Y}X\overline{W}}, \ \ \ \ \ \mathrm{Ric}(Q)_{X\overline{Y}} = \sum_{i=1}^n \big( R^b_{X\overline{Y}e_i\overline{e_i}} - R^b_{e_i\overline{Y}X\overline{e_i}} \big)
\end{equation*}
where $X,Y,Z$ and $W$ are type $(1,0)$ tangent vectors on $M$ and $\{ e_1, \ldots , e_n\}$ is any local unitary frame. Denote by $\omega$ the K\"ahler form of $g$.
Recall that {\em Gauduchon's torsion $1$-form} is the global $(1,0)$-form $\eta$ on $M$ defined by $\partial \omega^{n-1} = - \eta \wedge \omega^{n-1}$.
The metric $g$ is said to be {\em balanced} if $\eta =0$, or equivalently, if $d \omega^{n-1} = 0$.
Let $\chi$ be the type $(1,0)$ vector field on $M$ dual to $\eta$, namely, $\langle Y, \overline{\chi}\rangle = \eta (Y)$ for any $Y$ where $g = \langle , \rangle $ is the metric, extended bi-linearly over ${\mathbb C}$.

With these notations specified, we can now state the main result of this article, which says that the BTP condition $\nabla^b T^b=0$ is equivalent to four equalities involving only $R^b$:

\begin{theorem} \label{theorem1.1}
Let $(M^n,g)$ be a Hermitian manifold. The BTP condition $\nabla^bT^b=0$ is equivalent to the combination of the following
\begin{eqnarray}
&& R^b_{XYZ\,\overline{W}}=0 ,\label{eq:1.1} \\
&& R^b_{X\,\overline{Y}Z\,\overline{W}}=R^b_{Z\,\overline{W}X\overline{Y}} ,\label{eq:1.2} \\
&& \nabla^b \mathrm{Ric}(Q)=0, \label{eq:1.3}\\
&& \mathrm{Ric}(Q)_{\chi \overline{Y}} =0,  \label{eq:1.4}
\end{eqnarray}
for any type $(1,0)$ tangent vectors $X,Y,Z,W$.
\end{theorem}

\begin{remark} \label{remark1.2}
Theorem \ref{theorem1.1} can be regarded as a generalization of the main result of  \cite{ZhaoZ19Str}, which is a confirmation of the AOUV Conjecture.
\end{remark}

To be more precise, first let us recall that (cf. \cite{AOUV,FT,FTV,YZZ,ZhaoZ19Str,ZhaoZ23}) a Hermitian manifold is said to be {\em  Bismut K\"ahler-like} (or {\em BKL} for brevity), if its Bismut curvature $R^b$ obeys all K\"ahler symmetries, namely,
\begin{equation*}
 R^b_{XYZ\,\overline{W}}=0, \ \ \ \  R^b_{X\,\overline{Y}Z\,\overline{W}}=R^b_{Z\,\overline{Y}X\,\overline{W}}
\end{equation*}
for any type $(1,0)$ tangent vector $X,Y,Z,W$. A conjecture raised by Angella, Otal, Ugarte and Villacampa states  that:

\begin{conjecture}[{\bf AOUV \cite{AOUV}}] \label{conj1.3}
Compact BKL manifolds must be pluriclosed, namely, BKL $\Longrightarrow$ pluriclosed.
\end{conjecture}

The metric $g$ is said to be {\em pluriclosed} if $\partial \overline{\partial} \omega =0$, where $\omega$ is the K\"ahler form of $g$. See \cite{BFG,BPT,FG,Ye,IS,Yang} and the references therein for recent development regarding to pluriclosed metrics and related topics. Note that the BKL condition means (\ref{eq:1.1}) plus $Q=0$. In this case all the four conditions (\ref{eq:1.1})\,--\,(\ref{eq:1.4}) are satisfied. So as an immediate corollary to Theorem \ref{theorem1.1}, we have ``BKL $\Longrightarrow$ BTP". On the other hand, as noticed in  \cite{ZhaoZ19Str}, it is easy to show that ``BTP $\Longleftrightarrow $ pluriclosed" under the BKL assumption, and ``BTP + pluriclosed $\Longrightarrow $ BKL". Hence, as a direct consequence of Theorem \ref{theorem1.1}, we have the following

\begin{corollary} \label{cor1.4}
It holds on Hermitian manifolds that BKL $=$ BTP $+$ pluriclosed.
\end{corollary}

This is the main result of \cite{ZhaoZ19Str} even without the compactness assumption.
We remark that in Theorem \ref{theorem1.1}, the conditions \eqref{eq:1.3} and \eqref{eq:1.4}
are automatically satisfied if $\mathrm{Ric}(Q)=0$, which is analyzed in Proposition \ref{RicQ} below.
However, since there are examples of non-K\"ahler BKL manifolds with $\nabla^b R^b\neq 0$,
we know that for BTP manifolds, the Bismut connection $\nabla^b$ is not always an Ambrose-Singer
connection (cf. \cite{AS,NZ1,NZ2,BP}).

%\vspace{0.2cm}

Next, let us denote by $\nabla^c$ the Chern connection of $(M^n,g)$, and by $T=T^c$, $R^c$ its torsion and curvature. Under any local unitary tangent frame $\{ e_1, \ldots , e_n\}$, we have $T(e_i,\overline{e}_k)=0$ and $T(e_i,e_k)=\sum_j T^j_{ik} e_j$. These $T^j_{ik}$ will be called the {\em Chern torsion components} under the frame $e$ from now on. From $T$ we obtain the following $2$-tensors:
$$ B_{i\bar{j}}=\sum_{r,s} T^j_{rs} \overline{ T^i_{rs} } , \ \ \ \ A_{i\bar{j}} =\sum_{r,s} T^r_{is} \overline{T^r_{js}}, \ \ \ \ \phi^j_i = \sum_r T^j_{ir} \overline{\eta}_r. $$
Note that $A$, $B$, $\phi$ are well-defined global $2$-tensors on the manifold.  We have the following:

\begin{proposition} \label{prop1.5}
Let $(M^n,g)$ be a BTP manifold. Then under any local unitary frame $e$, the following hold:
\begin{eqnarray}
&&  \sum_r \big( T^r_{ij}T^{\ell}_{r k} + T^r_{jk}T^{\ell}_{ri} + T^r_{ki}T^{\ell}_{rj}\big) =0,  \ \ \ \ \ \forall \ 1\leq i,j,k,\ell \leq n, \label{eq:1.5} \\
&& 2 \sum_{r,s} T^r_{si} A_{r\bar{s}} =  \sum_{r,s} T^r_{si} B_{r\bar{s}} , \ \ \ \ \ \forall \ 1\leq i\leq n, \label{eq:1.6} \\
&& Q_{i\bar{j}k\bar{\ell}}  = \sum_r \big( T^j_{kr} \overline{ T^i_{\ell r} } + T^{\ell}_{ir} \overline{ T^k_{j r} } - T^j_{ir} \overline{ T^k_{\ell r} } - T^{\ell}_{kr} \overline{ T^i_{j r} } - T^r_{ik} \overline{ T^r_{j\ell } } \big) ,\ \ \ \ \ \forall \ 1\leq i,j,k,\ell \leq n, \label{eq:1.7} \\
&& \partial \eta =0, \ \ \ \overline{\partial } \eta = - \sum_{i,j} \overline{\phi^i_j }\,\varphi_i \wedge \overline{\varphi}_j, \label{eq:1.8}\\
&& R^b_{i\bar{j}k\bar{\ell}} - R^c_{i\bar{j}k\bar{\ell}} = \sum_r \big( T^{\ell}_{ir} \overline{T^{k}_{jr}} -T^{r}_{ik} \overline{T^{r}_{j\ell }} - T^{j}_{ir} \overline{T^{k}_{\ell r}} - T^{\ell}_{kr} \overline{T^{i}_{jr}}\big) , \ \ \ \ \forall \ 1\leq i,j,k,\ell \leq n, \label{eq:1.9}\\
&& \sum_r \big( \phi^r_i T^j_{rk} + \phi^r_k T^j_{ir} - \phi^j_r T^r_{ik} \big) = 0, \ \  \ \ \forall \ 1\leq i,j,k\leq n, \label{eq:1.10}\\
&& A, \,B, \,\phi, \,\phi^{\ast} \, \mbox{commute with each other}. \label{eq:1.11}
\end{eqnarray}
Here $\phi^{\ast}$ stands for $ (\phi^{\ast})^j_i = \overline{\phi^i_j}$, $\varphi$ is the coframe dual to $e$, while $R^b$ and $R^c$ denote the curvature tensor of the Bismut and Chern connection, respectively.
\end{proposition}

Since $\eta_i =\sum_k T^k_{ki}$ for any $i$, the equalities (\ref{eq:1.5})\,--\,(\ref{eq:1.7}) imply
\begin{equation} \label{eq:1.12}
\sum_r \eta_r T^r_{ik} = 0, \quad \forall\  1 \leq i,k \leq n;
\qquad 2\phi A = \phi B;
\qquad  \nabla^b Q = 0;
\qquad \mathrm{Ric}(Q)= B - \phi - \phi^{\ast} .
\end{equation}
Here $\phi A =\sum_{i,j} \phi^i_j A_{i\bar{j}} $,  $\phi B =\sum_{i,j} \phi^i_j B_{i\bar{j}} $.
The first equality in (\ref{eq:1.12}) is obtained from (\ref{eq:1.5}) after we let $j=\ell$ and sum up from $1$ to $n$, while the last equality in (\ref{eq:1.12})
is derived from (\ref{eq:1.7}) when we let $k=\ell$ and sum up. From the equality (\ref{eq:1.8}) and the defining equation of $\eta$,
one gets
\begin{equation} \label{eq:1.13}
\partial \overline{\partial} (\omega^{n-1}) =  \big( \eta \wedge \overline{\eta} + \overline{\partial} \eta \big) \wedge \omega^{n-1} = 0,
\end{equation}% \mbox{BTP}  \ \ \  \Longrightarrow \ \ \
that is, {\em any BTP metric is always a Gauduchon metric.} Similarly, as a consequence to (\ref{eq:1.11}), we get the existence of a special
local frames on any {\em non-balanced} BTP manifolds, analogous to the BKL case.

\begin{definition}
Let $(M^n,g)$ be a non-balanced BTP manifold.
Then a local unitary frame $e$ is called an {\bf admissible frame, }
if there exist global constants $\lambda >0$, $a_1, \ldots , a_{n-1},a_n$
with $a_1+\cdots +a_{n-1}=\lambda$ and $a_n=0$ such that
$$ \eta = \lambda \varphi_n, \ \ \ T^n_{ij}=0, \ \ \ T^j_{in}=a_i\delta_{ij}, \ \ \ \forall \ 1\leq i,j\leq n. $$
\end{definition}
Here $\varphi$ is the coframe dual to $e$. Note that $\lambda =|\eta| $ is a global constant
and $\lambda >0$ as $g$ is not balanced. Also, since $\lambda a_i$ are eigenvalues of the $\nabla^b$-parallel tensor $\phi$,
we know that $\{ a_1, \ldots , a_{n-1}, a_n\}$ is automatically a set of global constants.
Under an admissible frame $e$, the equation (\ref{eq:1.10}) now takes the form
\begin{equation} \label{eq:1.14}
(a_i+a_k-a_j)T^j_{ik}=0 , \ \ \ \ \forall \ 1\leq i,j,k\leq n.
\end{equation}

\begin{proposition} \label{prop1.7}
Let $(M^n,g)$ be a non-balanced BTP manifold and $\chi$ be the dual vector field of the Gauduchon's torsion $1$-form $\eta$.
Then locally there always exist admissible frames and $\chi$ is a holomorphic vector field with constant norm.
The metric $g$ is locally conformally balanced if and only if all $a_i$ are real number.
It can not be globally conformally balanced if $M$ is compact.
The restricted holonomy group of the Bismut connection is contained in $U(n-1) \times 1$.
\end{proposition}

Here {\em locally conformally balanced} means that $d(\eta +\overline{\eta})=0$.
In this case, any point on the manifold is contained in a neighborhood in which the metric $g$ is conformal to a balanced metric.
Also recall that a Hermitian manifold $(M^n,g)$ is said to be {\em locally conformally K\"ahler}
if its K\"ahler form $\omega$ satisfies $d\omega = \psi \wedge \omega$ for some closed global $1$-form on $M^n$,
called the {\em Lee form} of $g$. Clearly, $\psi$ is necessarily equal to $-\frac{1}{n-1}(\eta + \overline{\eta})$.
The metric $g$ is said to be {\em Vaisman} if $g$ is {\em locally conformally K\"ahler} and $\psi$ is parallel under the Levi-Civita connection.
Then we have the following results of Vaisman manifolds in terms of the torsion of Bismut connection,
where the first statement has recently been proved by Andrada and Villacampa \cite[Theorem 3.7]{AndV},
which indicates that all Vaisman manifolds are BTP. We will include a proof here for readers' convenience.

\begin{proposition}\label{prop1.8}
It holds that
\begin{enumerate}
\item $($\cite{AndV}$)$ for a locally conformally K\"ahler manifold $(M^n,g)$, it is BTP if and only if it is Vaisman.
\item for a non-balanced BTP manifold $(M^n,g)$, $g$ is Vaisman if and only if $a_1=\cdots =a_{n-1}=\frac{\lambda}{n-1}$.
\end{enumerate}
\end{proposition}

In particular, the standard Hopf manifold in all dimensions are BTP.
Ornea and Verbitsky \cite{OV03} proved a beautiful structure theorem for compact Vaisman manifolds,
which states that such a manifold must be a fiber bundle over a circle with fiber being a compact Sasakian manifold.
Since Vaisman manifolds are BTP, the last statement in Proposition \ref{prop1.7} extends another result \cite[Corollary 3.8]{AndV}
by Andrada and Villacampa from the Vaisman case to the general non-balanced BTP case.

After this, we will prove a more technical (local) result
which shows that BTP metrics are unique (if exist) within each conformal class, unless the metric is already locally conformally K\"ahler.
Note that when the manifold is compact, this is clearly true since BTP metrics are Gauduchon.

\begin{proposition} \label{prop1.9}
If two conformal Hermitian manifolds $(M^n,g)$ and $(M^n,e^{2u}g)$ are both BTP, with $du\neq 0$ in an open dense subset of $M$,
then  both manifolds are locally conformally K\"ahler $($hence Vaisman$)$.
\end{proposition}

Of course in the locally conformally K\"ahler case, locally there could be plenty of Vasiman metrics within the conformal class.
See Lemma \ref{LCK_VSM} in \S \ref{BTP_mfd} for more detailed discussion on this point.

When $n=2$, it was shown in \cite{ZhaoZ19Str} that BTP = BKL = Vaisman.
Compact Vaisman surfaces are fully classified by Belgun \cite{B00}.
When $n\geq 3$, on the other hand, it was shown in \cite{YZZ} that
(non-K\"ahler) BKL metrics can never be locally conformally K\"ahler
so BKL and Vaisman form disjoint subsets in BTP when $n\geq 3$.

In dimension $3$ and higher, there are plenty of BTP manifolds that are neither BKL nor Vaisman, and there are also balanced BTP manifolds. For instance, any complex semisimple Lie group $G$ equipped with the canonical metric (coming from the Killing form) is BTP. So the compact quotients of $G$ are balanced Chern flat BTP manifolds, and they are neither BKL nor Vaisman.

As concrete examples, one can classify all BTP manifolds in the category of nilmanifolds endowed with nilpotent complex structures, analogous to the BKL case studied in \cite{ZhaoZ19Nil}:

\begin{proposition}\label{prop1.10}
Let $(M^n,g)$ be a complex nilmanifold, namely, a compact Hermitian manifold whose universal cover is a nilpotent Lie group $G$ equipped with a left-invariant complex structure $J$ and a compatible left-invariant metric $g$. Assume that $J$ is nilpotent in the sense of \cite{CFGU}. Then $g$ is BTP if and only if there exists a unitary left-invariant coframe $\varphi$ on $G$, an integer $1\leq r\leq n$ and some coefficients $Y_{\alpha i}$ such that
\begin{equation*}
\begin{cases}
d\varphi_i=0,   &\forall \ 1\leq i\leq r;\  \\
d\varphi_{\alpha} = \sum\limits_{i=1}^r Y_{\alpha i}\,\varphi_i \wedge \overline{\varphi}_i,  & \forall\ r<\alpha \leq n.
\end{cases}
\end{equation*}
Note that the metric $g$ will be balanced when and only when $\sum_{i=1}^r Y_{\alpha i}=0$ for each $\alpha$,
$g$ will be BKL when and only when $\sum_{\alpha=r+1}^n Y_{\alpha i}\overline{Y_{\alpha k}} + \overline{Y_{\alpha i}} Y_{\alpha k} =0$
for $1 \leq i < k \leq r$,
and $g$ will be Vaisman when and only when $r = n-1$ and $Y_{ni}$ are equal for each $i$, or $r=n$.
\end{proposition}

In particular, $G$  (when not abelian) is a 2-step nilpotent group and $J$ is abelian. Let us take a closer look at the $n=3$ case.
In this case, either $r=3$, where $G$ is abelian and $g$ is K\"ahler and flat, or $r=2$ and there exists a unitary coframe $\varphi$ such that
$$ d\varphi_1=d\varphi_2=0, \ \ d\varphi_3 = a\varphi_1\wedge \overline{\varphi}_1 + b \varphi_2\wedge \overline{\varphi}_2, $$
where $a$, $b$ are arbitrary constants. By a unitary constant change of frame if necessary, we may assume that $a>0$.
Under such a frame, the metric $g$ is balanced if and only if $b=-a$, $g$ is BKL if and only if $b\in \sqrt{-1}{\mathbb R}$ is pure imaginary,
while $g$ is Vaisman if and only if $b=a$. For any other choice of $b$ values, $g$ is a non-balanced, non-BKL, non-Vaisman BTP metric.
Hence, the above result suggests that, when $n\geq 3$, there are still many BTP manifolds other than BKL manifolds and Vaisman manifolds,
and there are more non-balanced BTP manifolds than balanced ones. In a recent work \cite{PodestaZ}, Podest\`a and the second named author examined BTP metrics amongst Chern flat manifolds and flag varieties.

%\vspace{0.2cm}

Then let us recall that a Hermitian manifold is said to be {\em strongly Gauduchon} in the sense of Popovici \cite{Pop} if there exists a global $(n,n-2)$-form $\Omega$ such that $\partial \omega^{n-1} = \overline{\partial} \Omega$.
The implications hold clearly: $\, \mbox{balanced} \,  \Longrightarrow \,  \mbox{strongly\ Gauduchon}  \, \Longrightarrow  \, \mbox{Gauduchon}$. As a consequence to Theorem \ref{theorem1.1}, we have

%{\em Gauduchon} if $\partial \overline{\partial} (\omega^{n-1})=0$, and is said to be {\em balanced} if its Gauduchon's torsion $1$-form $\eta$ vanishes, or equivalently, if $d(\omega^{n-1})=0$. Also, the metric is said to be

\begin{corollary} \label{cor1.11}
Let $(M^n,g)$ be a compact non-balanced BTP manifold. Then $g$ can not be strongly Gauduchon, and the Dolbeault group $H^{0,1}_{\overline{\partial}}(M) \neq 0$.
In particular $M$ does not satisfy the $\partial \overline{\partial}$-lemma and thus it does not admit any K\"ahler metric.
\end{corollary}

It was shown in \cite[Theorem 4]{ZhaoZ19Str} that any compact non-K\"ahler BKL manifold does not admit any strongly Gauduchon metric. We speculate that the same should be true for compact non-balanced BTP manifolds, and we propose the following

\begin{conjecture}\label{conj1.12}
Let $(M^n,g)$ be a compact non-balanced BTP manifold. Then $M$ does not admit any strongly Gauduchon metric.
\end{conjecture}

Note that the above conjecture is valid when $B \geq \mbox{Ric}(Q)$, as observed in the discussion after the proof of Corollary \ref{cor1.11}.
It is easy to verify $B \geq \mbox{Ric}(Q)$ holds for Vaisman manifolds, thus Conjecture \ref{conj1.12} is true for all Vaisman manifolds.
This special case has been confirmed by Angella-Otiman \cite{AO}. The above conjecture could also be regarded as a variant
of two famous conjectures by Streets-Tian and Fino-Vezzoni. Here we recall that a Hermitian metric g is called {\em Hermitian symplectic,}
if its K\"ahler form $\omega$ satisfies $\partial \omega = -\overline{\partial} \alpha$ and $\partial \alpha =0$ for some $(2,0)$-form $\alpha$.
Equivalently, there exists a $(2,0)$-form $\alpha$ on the manifold such that $d(\alpha + \omega + \overline{\alpha })=0$.
Such a metric is always pluriclosed, namely, $\partial\overline{\partial}  \omega=0$.

\begin{conjecture}[{\bf Streets-Tian \cite{ST10}}]
If a compact complex manifold $M^n$ admits a
Hermitian-symplectic metric, then it must be K\"ahlerian (namely, it admits a K\"ahler metric).
\end{conjecture}

In other words, compact non-K\"ahlerian manifolds can not admit any Hermitian-symplectic metric. They confirmed the above conjecture for $n=2$.

\begin{conjecture}[{\bf Fino-Vezzoni \cite{FV2}}]
If a compact complex manifold $M^n$ admits a pluriclosed metric $g$ and a balanced metric $h$,
then it must admit a K\"ahler metric.
\end{conjecture}

As another consequence of Theorem \ref{theorem1.1}, it intrigues us to study whether
$\mathrm{Ric}(Q)=0$ on a BTP manifold, which means  that the three Bismut Ricci curvatures coincide,
actually implies $Q=0$, that is, the manifold becomes BKL and Bismut curvature obeys K\"ahler symmetries.
It is clear from the last equality in \eqref{eq:1.12} that this question only concerns the non-balanced case,
since in the balanced case the assumption $\mathrm{Ric}(Q)=0$ would force $B=0$ thus $g$ is K\"ahler.
Utlizing Theorem \ref{theorem1.1}, Proposition \ref{prop1.5}, Proposition \ref{prop1.7}
and Proposition \ref{prop1.10}, we obtain the following result

\begin{proposition}\label{RicQ}
Let $(M^n,g)$ be a non-balanced BTP manifold satisfying $\mathrm{Ric}(Q)=0$.
When $n=3,4$, it holds that $Q=0$, while, there exists a BTP nilmanifold of complex dimension $5$ satisfying $\mathrm{Ric}(Q)=0$ but $Q \neq 0$.
\end{proposition}

The second main result  of this article is the following, which gives the structure (classification) of non-balanced BTP threefolds:

\begin{theorem}\label{3DNBBTP}
Let $(M^3,g)$ be a complete non-balanced BTP threefold. If we denote by $\{ \lambda a_1, \lambda a_2,0\}$ the three eigenvalues of the tensor $\phi$, then the following hold:
\begin{enumerate}
\item\label{BKL} when $a_1\overline{a}_2+\overline{a}_1a_2 = 0$, $(M^3,g)$ is BKL, and the universal cover $\widetilde{M}$ is holomorphically isomorphic to either a product of a BKL surface and a K\"ahler curve, or a product $N_1 \times N_2$ of two Sasakian $3$-manifolds ,
\item\label{twSSK} when $a_1\overline{a}_2+\overline{a}_1a_2 \neq 0$ and $a_1\overline{a}_2-\overline{a}_1a_2 \neq 0$, the universal cover $\widetilde{M}$ of $M^3$ is holomorphically isomorphic to a twisted product $N_1 \times_{\kappa} N_2 $ of two Sasakian $3$-manifolds,

\item\label{gV} the remaining case is when both $a_1$ and $a_2$ are real and non-zero, in this case $(M^3,g)$ is locally conformally balanced, with its universal cover $\widetilde{M}$ holomorphically isomorphic to a product $N \times \mathbb{R}$ of a generalized Sasakian $5$-manifold and the real line.
\end{enumerate}
\end{theorem}
The {\em twisted product of Sasakian manifolds} is the modified Hermitian structure on the product of two Sasakian manifolds, introduced by Belgun \cite{B12}, which we will recall in Definition \ref{SSK} and \ref{SSKproduct}.
The notion of {\em generalized Sasakian manifold} will also be recalled in \S \ref{NBBTP}. Note that the last case in Theorem \ref{3DNBBTP} include all Vaisman threefolds, which correspond to the case of $a_1=a_2=\frac{\lambda }{2}$.

We will end the article with a couple of interesting corollaries of Theorem \ref{3DNBBTP}. To state the first one, let us introduce a terminology:

\begin{definition}
A Hermitian manifold $(M^n,g)$ is said to be {\em Bismut $($holonomy$)$ abelian,}
if the global holonomy group $\mbox{Hol}(\nabla^b)$ is an abelian group.
\end{definition}

Since any abelian subgroup of $U(n)$ is conjugated to a subgroup of $U(1)^n=U(1)\times \cdots \times U(1)$, the above definition implies locally there always exists unitary frame $e$ under which the Bismut connection matrix is diagonal. The Bismut curvature matrix under $e$ is also diagonal.

\begin{corollary} \label{cor1.16}
Let $(M^3,g)$ be a non-balanced BTP threefold. Then it is Bismut abelian unless it is Vaisman.
\end{corollary}

Actually in the above case the global Bismut holonomy group $\mbox{Hol}(\nabla^b)$ is contained in $U(1)\times U(1)\times 1$ since $g$ is non-balanced BTP. We remark that a Vaisman threefold could also be Bismut abelian, even though the generic ones (including the standard Hopf threefold) is not this case. Bismut abelian Vaisman threefolds \cite[Subsection 5.3]{ZhaoZ25} form a highly interesting subclass, which arise a little surprisingly in the study of balanced BTP threefolds. However, we will see that the conclusion of Corollary \ref{cor1.16} would fail when the BTP threefold is allowed to be balanced in a forthcoming project, in which we will give a classification theorem for compact balanced (non-K\"ahler) BTP threefolds, utilizing a technical lemma obtained in \cite{ZhouZ} which says that any balanced theefold always admits special local frames under which the Chern torsion components take a particularly simple form.

The second corollary of Theorem \ref{3DNBBTP} is the following:

\begin{corollary} \label{cor1.17}
Let $(M^3,g)$ be a compact non-balanced BTP threefold. If $M^3$ admits a pluriclosed metric, then $g$ is BKL.
In particular, the Conjectures of Street-Tian and Fino-Vezzoni are both valid for compact non-balanced BTP threefolds.
\end{corollary}

The last sentence is due to the fact that any compact non-K\"ahler (or equivalently, non-balanced) BKL manifold can not admit any Hermitian symplectic metric by \cite[Theorem 1]{YZZ} or any balanced (or strongly Gauduchon) metric by \cite[Theorem 3]{ZhaoZ19Str}. Recently,
Streets-Tian conjecture has been confirmed on compact non-balanced BTP manifolds of any complex dimensions, by Y. Guo and
the second named author in \cite{GZ}. In the sequel we will show that given any compact balanced BTP threefold, either it is already K\"ahlerian, or it does not admit any pluriclosed metric. So both conjectures hold for compact balanced BTP threefolds as well.
So in summary we know that {\em both the Streets-Tian Conjecture and the Fino-Vezzoni Conjecture hold for all compact BTP threefolds}.

This article grow from the first half of the preprint \cite{ZhaoZ22} which we put on arXiv about two years ago.
The second half will be revised to give a classification of compact balanced BTP threefolds as a subsequent study \cite{ZhaoZ25}.

%\vspace{0.3cm}

\vspace{0.4cm}

\section{Torsion and curvature of Bismut connection}\label{PRE}

First let us set up the notations and fix some terminologies. Let $(M^n,g)$ be a Hermitian manifold of complex dimension $n$.
We will follow the notations of  \cite{YZ18Cur,YZ18Gau,ZhaoZ19Str, Zheng, Zheng1}. Let $g = \langle , \rangle $ be the metric, extended bi-linearly over ${\mathbb C}$.
The bundle of complex tangent vector fields of type $(1,0)$, namely, complex vector fields of the form $v- \sqrt{-1}Jv$, where $v$ is a real vector field on $M$, is denoted by $T^{1,0}M$.
Let $\{ e_1, \ldots , e_n\}$ be a local frame of $T^{1,0}M$ in a neighborhood of $M$. Write $e=\ ^t\!(e_1, \ldots , e_n) $ as a column vector.
Denote by $\varphi = \ ^t\!(\varphi_1, \ldots ,  \varphi_n)$ the column vector of local $(1,0)$-forms which is the coframe dual to $e$, namely, $\varphi_i(e_j)=\delta_{ij}$ for any $1\leq i,j\leq n$.

Denote by $\nabla$, $\nabla^c$, $\nabla^b$ the Levi-Civita (Riemannian), Chern, and Bismut connection, respectively. Denote by $T^c=T$, $R^c$ the torsion and curvature of $\nabla^c$, and by $T^b$, $R^b$ the torsion and curvature of $\nabla^b$. Under the frame $e$, denote the components of $T^c$ as
\begin{equation}
T^c(e_i, \overline{e}_j)=0, \ \ \ T^c(e_i, e_j) =  \sum_{k=1}^n T^k_{ij} e_k.  \label{eq:2.1}
\end{equation}
Note that our $T^j_{ik}$ here is twice as much as the same notation in \cite{YZ18Cur, ZhaoZ19Str},
which will yield some changes in the coefficients.

For Chern connection $\nabla^c$, let us denote by $\theta^c=\theta$,  $\Theta^c=\Theta$ the matrices of connection and
curvature, respectively, and by $\tau$ the column vector of the
torsion $2$-forms, all under the local frame $e$. Then the structure
equations and Bianchi identities are
\begin{eqnarray*}
d \varphi & = & - \ ^t\!\theta \wedge \varphi + \tau,  \label{formula 1}\\
d  \theta & = & \theta \wedge \theta + \Theta. \\
d \tau & = & - \ ^t\!\theta \wedge \tau + \ ^t\!\Theta \wedge \varphi, \label{formula 3} \\
d  \Theta & = & \theta \wedge \Theta - \Theta \wedge \theta.
\end{eqnarray*}
The entries of $\Theta$ are all $(1,1)$ forms, while the entries of the column vector $\tau $ are all $(2,0)$ forms, under the frame $e$.
Similar symbols such as $\theta^b,\Theta^b$ and $\tau^b$ are applied to Bismut connection.
The components of $\tau$ are exactly $T_{ij}^k$ defined in (\ref{eq:2.1}):
\[ \tau_k =\frac{1}{2}  \sum_{i,j=1}^n T_{ij}^k \varphi_i\wedge \varphi_j \ = \sum_{1\leq i<j\leq n}  \ T_{ij}^k \varphi_i\wedge \varphi_j.\]
Under the frame $e$, express the Levi-Civita connection $\nabla$ as
$$ \nabla e = \theta_1 e + \overline{\theta_2 }\overline{e} ,
\ \ \ \nabla \overline{e} = \theta_2 e + \overline{\theta_1
}\overline{e} ,$$
thus the matrices of connection and curvature for $\nabla $ become:
$$ \hat{\theta } = \left[ \begin{array}{ll} \theta_1 & \overline{\theta_2 } \\ \theta_2 & \overline{\theta_1 }  \end{array} \right] , \ \  \  \hat{\Theta } = \left[ \begin{array}{ll} \Theta_1 & \overline{\Theta}_2  \\ \Theta_2 & \overline{\Theta}_1   \end{array} \right], $$
 where
\begin{eqnarray*}
\Theta_1 & = & d\theta_1 -\theta_1 \wedge \theta_1 -\overline{\theta_2} \wedge \theta_2, \\
\Theta_2 & = & d\theta_2 - \theta_2 \wedge \theta_1 - \overline{\theta_1 } \wedge \theta_2,  \label{formula 7}\\
d\varphi & = & - \ ^t\! \theta_1 \wedge \varphi - \ ^t\! \theta_2
\wedge \overline{\varphi } .
\end{eqnarray*}
When $e$ is unitary, both $\theta_2 $ and $\Theta_2$ are skew-symmetric, while $\theta$, $\theta_1$, $\theta^b$, or $\Theta$, $\Theta_1$, $\Theta^b$ are all skew-Hermitian. Consider the $(2,1)$ tensor $\gamma = (\nabla^b -\nabla^c)$ introduced in \cite{YZ18Cur}. Its representation under the frame $e$ is a matrix of $1$-forms, which by abuse of notation we will also denote by $\gamma$. Then it holds from \cite[Lemma 2]{WYZ} that
\[ \frac{1}{2}\gamma = \theta_1 - \theta .\]
Denote the decomposition of $\gamma$ into $(1,0)$ and $(0,1)$ parts by $\gamma = \gamma ' + \gamma ''$.
As observed in \cite{YZ18Cur}, when $e$ is unitary, $\gamma $ and $\theta_2$ take the following simple forms
\begin{equation}
(\theta_2)_{ij} = \frac{1}{2}\sum_{k=1}^n \overline{T^k_{ij}} \varphi_k, \ \ \ \ \gamma_{ij} = \sum_{k=1}^n ( T_{ik}^j \varphi_k - \overline{T^i_{jk}} \overline{\varphi}_k ), \label{eq:2.6}
\end{equation}
while for general frames the above formula will have the matrix $(g_{i\overline{j}})=(\langle e_i,\overline{e}_j \rangle )$ and its inverse involved, which is less convenient. This is why we often choose unitary frames to work with. By (\ref{eq:2.6}), we get the expression of the components of the torsion $T^b$ under any unitary frame $e$:
\begin{equation}
T^b(e_i,e_j) = - \sum_{k=1}^n T^k_{ij}e_k, \ \ \ T^b(e_i, \overline{e}_j) = \sum_{k=1}^n \big( T^j_{ik}\overline{e}_k - \overline{T^i_{jk}} e_k \big) . \label{eq:2.7}
\end{equation}
From this, we deduce that
\begin{equation}
\nabla^bT^b=0 \ \ \Longleftrightarrow \ \ \nabla^bT^c =0. \ \label{eq:2.8}
\end{equation}
This is observed  in \cite[the proof of Proposition 1]{ZhaoZ19Str}. More generally, if we denote by $\nabla^t = (1-t)\nabla^c + t \nabla^b$ the $t$-Gauduchon connection, where $t$ is any real constant, then its torsion tensor is $T^t(x,y)=T^c(x,y)+ t(\gamma_xy -\gamma_yx)$, whose components under $e$ again can be expressed in terms of  $T^k_{ij}$, so it is easy to see that for any fixed $t\in {\mathbb R}$,
\begin{equation}
\nabla^tT^{t_1}=0 \ \ \Longleftrightarrow  \ \ \nabla^tT^{t_2} =0, \ \ \ \ \forall \  t_1, t_2 \in \mathbb{R}. \ \label{eq:2.9}
\end{equation}
That is, for any fixed $t$, if the torsion of one Gauduchon connection is parallel with respect to $\nabla^t$, then the torsion of any other Gauduchon connection will also be parallel under $\nabla^t$.

Next let us discuss the curvature. As usual, the curvature tensor $R^D$ of a linear connection $D$ on a Riemannian manifold $M$ is defined  by
$$ R^D(x,y,z,w) = \langle R^D_{xy}z, \ w \rangle = \langle D_xD_yz-D_yD_xz - D_{[x,y]}z, \ w \rangle, $$
where $x,y,z,w$ are tangent vectors in $M$. We will also write it as $R^D_{xyzw}$ for brevity. It is always skew-symmetric with respect to the first two positions, and it will be skew-symmetric with respect to its last two positions when the connection is {\em metric,}  namely, when $Dg=0$.  The first and second Bianchi identities are respectively
\begin{eqnarray*}
&& {\mathfrak S}\{ R^D_{xy}z - (D_xT^D)(y,z) - T^D(T^D(x,y),z)  \} = 0,   \label{eq:B1} \\
&& {\mathfrak S}\{ (D_xR^D)_{yz} + R^D_{T^D(x,y)\, z} \} = 0,  \label{eq:B2}
\end{eqnarray*}
where ${\mathfrak S}$ means the sum over all cyclic permutation of $x,y,z$. When $M$ is a complex manifold and the connection $D$ is {\em Hermitian,} namely, satisfying $Dg=0$ and $DJ=0$ where $J$ is the almost complex structure, then $R^D$ satisfies
$$ R^D(x,y,Jz,Jw) = R^D(x,y,z,w), $$
for any tangent vectors $x,y,z,w$, or equivalently,
\begin{equation}
 R^D(x,y,Z,W) = R^D(x,y,\overline{Z},\overline{W}) =0 , \label{eq:2.10}
 \end{equation}
for any type $(1,0)$ tangent vectors $Z$, $W$. From now on, we will use $X$, $Y$, $Z$, $W$ to denote type $(1,0)$ vectors. For any Hermitian connection $D$, with the skew-symmetry of first two or last two positions in mind, the above equation (\ref{eq:2.10}) says that the only possibly non-zero components of $R^D$ are $R^D_{XYZ\overline{W}}$, $R^D_{X\overline{Y}Z\overline{W}}$ and their conjugations.

Let us now apply the Bianchi identities to the Chern connection $\nabla^c$, and use the fact that $T^c(e_i, \overline{e}_j)=0$ and $R^c_{e_ie_j}x =0$ for any $x$ where $e$ is any local unitary frame. We have the following
\begin{eqnarray}
&&  T^{\ell}_{ij;k} + T^{\ell}_{jk;i}+ T^{\ell}_{ki;j}\, = \,  \sum_{r=1}^n \big( T^r_{ij}T^{\ell}_{kr} + T^r_{jk}T^{\ell}_{ir}+ T^r_{ki}T^{\ell}_{jr}\big)  \label{eq:B1a} \\
&&  R^c_{k\overline{j}i\overline{\ell}} -  R^c_{i\overline{j}k\overline{\ell}} \, = \,  T^{\ell}_{ik;\overline{j}} \label{eq:B1b} \\
&& R^c_{i\overline{j}k\overline{\ell};\, m} - R^c_{m\overline{j}k\overline{\ell}; \,i} \,= \, \sum_{r=1}^n T^{r}_{im} R^c_{r\overline{j}k\overline{\ell}}  \label{eq:B2}
\end{eqnarray}
where the indices after semicolon stand for covariant derivatives with respect to the Chern connection $\nabla^c$. Note that in the formula (\ref{eq:B1b}) we assumed that the frame $e$ is unitary, otherwise the right hand side needs to be multiplied on the right by $g$, the matrix of the metric.

%\vspace{0.3cm}

%\section{Torsion and curvature of Bismut connection}\label{TC}

Since we will be primarily interested in the Bismut connection $\nabla^b$, we would like to have formulae involving the curvature and covariant differentiation of $\nabla^b$. Using the Bianchi identities for general metric connections, we could get formulae for $\nabla^b$ which are similar to the ones for Chern connection mentioned above. But since we also need relations between $R^c$ and $R^b$ later, let us deduce them from the structure equations. Under a $(1,0)$-frame $e$, the components of the Chern, Bismut and Riemannian curvature tensors are given by
\begin{equation*}
R^c_{i\overline{j}k\overline{\ell}}  =  \sum_{p=1}^n \Theta_{kp}(e_i,
\overline{e}_j)g_{p\overline{\ell}}, \ \ \  \ \ R^b_{abk\overline{\ell}}  =  \sum_{p=1}^n \Theta^b_{kp}(e_a,
e_b)g_{p\overline{\ell}}, \ \ \ \ \ R_{abcd} = \sum_{f=1}^{2n} \hat{\Theta}_{cf}(e_a,e_b)g_{fd},
\end{equation*}
where $i,j,k,\ell,p$ range from $1$ to $n$, while $a,b,c,d,f$ range from $1$ to $2n$ with
$e_{n+i}=\overline{e}_i$, and $g_{ab}=g(e_a,e_b)$. It follows from our previous discussion that $R^D_{abij} = R^D_{ab\bar{i}\bar{j}} =0$ for any Hermitian connection $D$.

For convenience, we will assume from now on that our local $(1,0)$-frame $e$ is unitary for the Hermitian metric $g$. Then we have

\begin{lemma} \label{lemma3.1}
Let $(M^n,g)$ be a Hermitian manifold. Then under any local unitary frame the following identities hold:
\begin{eqnarray}
&& R^{b}_{ijk\bar{\ell}} \ = \
\big( T^{\ell}_{kj,i}-T^{\ell}_{ki,j} \big) +\sum_r \big( T^r_{ij}T^{\ell}_{rk} + T^r_{jk}T^{\ell}_{ri} + T^r_{ki}T^{\ell}_{rj} \big) ,       \label{curv20}\\
 && R^b_{i\bar{j}k\bar{\ell}}-R^c_{i\bar{j}k\bar{\ell}} \ = \ \big( T^{\ell}_{ik,\bar{j}} + \overline{T^{k}_{j\ell,\bar{i}}} \big)  -  \sum_r  \big( T^{\ell}_{kr}\overline{T^{i}_{jr}} + T^{j}_{ir}\overline{T^k_{\ell r}}
+ T^r_{ik}\overline{T^r_{j\ell}} - T^{\ell}_{ir}\overline{T^{k}_{jr}} \big),     \label{curv11}\\
&& T^k_{ij,\ell} + T^k_{j\ell,i} + T^k_{\ell i,j} \ = \ 2\sum_r  \big( T^r_{\ell i}T^k_{r j} + T^r_{j\ell}T^k_{ri} + T^r_{ij}T^k_{r \ell} \big) ,      \label{pT} \\
&& T^j_{ik,\bar{\ell}} + \overline{T^i_{j\ell,\bar{k}}} - \overline{T^k_{j\ell,\bar{i}}} \ \, = \ \,   -\, \sum_r \big( T^r_{ik}\overline{T^r_{j\ell}} + T^j_{ir} \overline{T^k_{\ell r}} - T^j_{kr}\overline{T^i_{\ell r}}
-T^\ell_{ir} \overline{T^k_{jr}} + T^\ell_{kr}\overline{T^i_{jr}} \big)       \label{dbT} \\
&& \hspace{4cm} + \ \,\big( R^{b}_{i\bar{\ell}k\bar{j}}-R^b_{k\bar{\ell}i\bar{j}} \big), \notag \\
&& R^{b}_{ijk\bar{\ell},p} + R^{b}_{pik\bar{\ell},j} +R^{b}_{jpk\bar{\ell},i} \ = \
 -\,\sum_r \big( R^b_{irk\bar{\ell}}T_{jp}^r + R^b_{jrk\bar{\ell}}T_{pi}^r + R^b_{prk\bar{\ell}}T_{ij}^r \big) , \label{dR_30}  \\
&& R^{b}_{ipk \bar{\ell},\bar{q}} - R^{b}_{i\bar{q}k \bar{\ell}, p} + R^{b}_{p \bar{q} k \bar{\ell}, i} \  = \ \sum_r \{  R^b_{r\bar{q}k\bar{\ell}}T^r_{ip}  - R^b_{p\bar{r}k\bar{\ell}}T^q_{ir}   + R^b_{i\bar{r}k\bar{\ell}}T^q_{pr} +
\label{dR_21} \\
&& \hspace{4.8cm} +  \ R^b_{prk\bar{\ell}} \overline{T^i_{qr}}  - R^b_{irk\bar{\ell}}\overline{T^p_{qr}}\} ,\notag \end{eqnarray}
for any $i,j,k,\ell,p,q$, where indices after comma stand for covariant derivatives with respect to $\nabla^{b}$.
\end{lemma}

\begin{proof}
For $\forall\ x\in M$, we will choose our local unitary frame $e$ near $x$ such that $\theta^b\big|_{x}=0$, thus $\theta^c=-\gamma$ at $x$.
From the structure equation $\, \Theta^b = d \theta^b - \theta^b \wedge \theta^b$,  it yields that
\[\Theta^b = \Theta +  (d \gamma +  \gamma \wedge \gamma) \]
at $x$. Taking the $(2,0)$ and $(1,1)$ parts of the above, we get respectively
\begin{eqnarray}
\Theta^b_{2,0}&=& (\partial \gamma' +  \gamma' \wedge \gamma'), \notag\\
\Theta^b_{1,1}&=& \Theta + (\overline{\partial} \gamma' - \partial\overline{^t\!\gamma'}
-  \gamma' \wedge \overline{^t\!\gamma'} - \overline{^t\!\gamma'} \wedge \gamma'),\notag
\end{eqnarray}
hence \eqref{curv20} and \eqref{curv11} are followed. From the first Bianchi identity $\, d \tau = - ^t\!\theta \wedge \tau + ^t\!\!\Theta \wedge \varphi$,
it yields at $x$ that
\begin{eqnarray}
\partial \tau&=&  ^t\!\gamma' \wedge \tau, \notag\\
\overline{\partial} \tau&=& -\overline{\gamma'} \wedge \tau + \,^t\!\Theta \wedge \varphi,\notag \\
&=& -\overline{\gamma'} \wedge \tau + \,^t\!\Theta^b_{1,1} \wedge \varphi
- \,^t\!\big(\overline{\partial} \gamma' - \partial\overline{^t\!\gamma'}
-  \gamma' \wedge \overline{^t\!\gamma'} - \overline{^t\!\gamma'} \wedge \gamma'\big) \wedge \varphi, \notag
\end{eqnarray}
which imply \eqref{pT} and \eqref{dbT}. By the second Bianchi identity, $d \Theta^b = \theta^b \wedge \Theta^b - \Theta^b \wedge \theta^b$,
we have $d \Theta^b=0$ at $x$ since  $\theta^b\big|_{x}=0$. Therefore
\[\partial \Theta^b_{2,0}=0,\quad \overline{\partial} \Theta^b_{2,0}+ \partial \Theta^{b}_{1,1}=0, \]
which lead to \eqref{dR_30} and \eqref{dR_21}.
\end{proof}

\begin{lemma} \label{lemma3.2}
Let $(M,g)$ be a Hermitian manifold. Then under any local unitary frame we have
\begin{eqnarray}
T^{\ell}_{ij,k} &= & -(R^b_{jki \bar{\ell}} + R^b_{kij \bar{\ell}}), \label{pT_refined}\\
\sum_r(T^r_{ij}T^{\ell}_{rk} + T^r_{jk}T^{\ell}_{ri} + T^r_{ki}T^{\ell}_{rj}) &=&
-(R^b_{ijk \bar{\ell}} + R^b_{jki \bar{\ell}} + R^b_{kij \bar{\ell}}), \label{permutation}\\
T^j_{ik,\bar{\ell}} & = & -\frac{1}{3} \sum_r(T^r_{ik}\overline{T^r_{j\ell}} + T^j_{ir} \overline{T^k_{\ell r}}
- T^j_{kr}\overline{T^i_{\ell r}} -T^\ell_{ir} \overline{T^k_{jr}} + T^\ell_{kr}\overline{T^i_{jr}}) \label{dbT_refined}\\
&&+ \frac{2}{3} (R^b_{i \bar{\ell} k \bar{j}} - R^b_{k \bar{\ell} i \bar{j}})
+ \frac{1}{3} (R^b_{i \bar{j} k \bar{\ell}} - R^b_{k \bar{j} i \bar{\ell}}), \notag\\
\eta_{i,j} & = & - \sum_r (R^b_{ijr\bar{r}} + R^b_{jri\bar{r}}), \label{peta}\\
\eta_{i,\bar{j}} & = & -\frac{1}{3}(\sum_{r}\eta_r \overline{T^i_{jr}} + \overline{\eta}_r T^j_{ir} - \sum_{t,r}T^j_{tr}\overline{T^i_{tr}}) \label{dbeta} \\
& & + \frac{2}{3}\sum_{r}(R^b_{r\bar{j}i\bar{r}} - R^b_{i\bar{j}r\bar{r}})
+ \frac{1}{3}\sum_r(R^b_{r\bar{r}i\bar{j}} - R^b_{i\bar{r}r\bar{j}}) \notag.
\end{eqnarray}
for any $i,j,k,\ell$, where indices after comma stand for the covariant derivative with respect to $\nabla^{b}$.
\end{lemma}

\begin{proof}
Equations \eqref{curv20} and \eqref{pT} imply  \eqref{pT_refined} and \eqref{permutation}, and equation \eqref{dbT} implies \eqref{dbT_refined}. Also, since $\eta_k=\sum_i T^i_{ik}$, \eqref{peta} and \eqref{dbeta} are established from \eqref{pT_refined} and \eqref{dbT_refined}, respectively.
\end{proof}

\begin{definition}[See the proof of Lemma 6 and Section 4 in \cite{ZhaoZ19Str}]\label{def3.3}
Let us introduce the following notations
\[P_{ik}^{j \ell}:=  \sum_r T^r_{ik}\overline{T^r_{j\ell}} + T^j_{ir} \overline{T^k_{\ell r}}
- T^j_{kr}\overline{T^i_{\ell r}} -T^\ell_{ir} \overline{T^k_{jr}} + T^\ell_{kr}\overline{T^i_{jr}}. \ \]
\[ A_{k\overline{\ell }} := \sum_{r,s} T^r_{sk} \overline{ T^r_{s\ell } } , \quad   \quad \ \  B_{k\overline{\ell }} := \sum_{r,s} T^{\ell }_{rs} \overline{ T^k_{rs } } ,  \quad  \quad  \quad  \quad \quad  \  \]
\[C_{ik} := \sum_{r,s} T^r_{si} T^s_{rk},  \quad  \quad  \quad  \phi^{\ell }_k := \sum_r \overline{\eta}_r T^{\ell }_{kr}, \quad  \quad  \quad  \quad  \quad  \quad    \]
\[Q_{i\bar{j}k\bar{\ell}} := R^b_{i\bar{j}k\bar{\ell}} - R^b_{k \bar{j} i \bar{\ell}},
\quad \mathrm{Ric}(Q)_{i\bar{j}} := \sum_r R^b_{i\bar{j}r\bar{r}} - R^b_{r \bar{j} i \bar{r}}\]
under any unitary frame. It is clear that $C$ is symmetric, $A$, $B$ are Hermitian symmetric, and $P_{ik}^{j \ell}$ satisfies
\[ P^{j\ell}_{\,ik} = - P^{j\ell }_{\,ki} = - P^{\ell j}_{\,ik} = \overline{ P^{ik}_{j\ell } }.\]
\end{definition}

\begin{proposition}\label{plcld}
Let $(M^n,g)$ be a Hermitian manifold and denote by $\omega$ the K\"ahler form of $g$. Then
\[\sqrt{-1}\partial \overline{\partial} \omega = \frac{1}{4} \sum_{i,j,k,\ell} \left\{ (T^{\ell}_{ik,\overline{j}} - T^j_{ik,\overline{\ell}})-P^{j\ell}_{ik} \right\} \varphi_i \wedge \varphi_k \wedge \overline{\varphi}_j \wedge \overline{\varphi}_{\ell}.\]
Hence we have the following equivalence
\[\begin{aligned}
\partial \overline{\partial} \omega=0 \quad & \Longleftrightarrow \quad T^{j}_{ik,\overline{\ell}} - T^{\ell}_{ik,\overline{j}}=-P^{j\ell}_{ik}\\
& \Longleftrightarrow \quad (R^b_{i \overline{\ell} k \overline{j}}-R^b_{k \overline{\ell} i \overline{j}})
- (R^b_{i \overline{j} k \ell} - R^b_{k \overline{j} i \overline{\ell}}) = - P^{j \ell}_{ik}.
\end{aligned}\]
\end{proposition}

\begin{proof}
By \cite[Lemma 1]{ZhaoZ19Str}, we have
\[\begin{aligned}
\sqrt{-1} \partial \overline{\partial} \omega &=  ^t\!\!\tau  \overline{\tau} + \, ^t\!\varphi \Theta \overline{\varphi}\\
&= \frac{1}{4} \sum_{i,j,k,\ell,p}\left(T^p_{ik}\overline{T^p_{j\ell}}
- (R^c_{i\bar{j}k\bar{\ell}} - R^c_{k\bar{j}i\bar{\ell}}- R^c_{i\bar{\ell}k\bar{j}}+R^c_{k\bar{\ell}i\bar{j}})\right)
\varphi_i \wedge \varphi_k \wedge \overline{\varphi}_j \wedge \overline{\varphi}_{\ell}.
\end{aligned}\]
It follows from the structure equation \eqref{eq:B1b} of Chern connection above
\begin{equation}\label{eq:DR}
T^j_{ik;\overline{\ell}}  =  R^c_{k\overline{\ell}i\overline{j}} - R^c_{i\overline{\ell}k\overline{j}}, \quad \forall\ i,j,k,\ell,
\end{equation}
where the index after the semicolon stands for covariant derivative with respect to the Chern connection $\nabla^c$, and it is easy to obtain,
since $\nabla^{c}-\nabla^b = -\gamma$,
\begin{eqnarray}
T^j_{ik ; \overline{\ell}} & = & \overline{e}_{\ell} T^j_{ik} - \sum_q \left(T^j_{qk}\,\langle  \nabla^{c}_{\overline{e}_{\ell} } e_i, \overline{e}_q \rangle
+ T^j_{iq}\,\langle   \nabla^{c}_{\overline{e}_{\ell} } e_k, \overline{e}_q \rangle
+ T^q_{ik}\, \langle  \nabla^{c}_{\overline{e}_{\ell} }  \overline{e}_j, e_q\rangle  \right) \nonumber \\
& = & T^j_{ik , \overline{\ell}} + \sum_q \left( T^j_{qk} \gamma_{iq}(\overline{e}_{\ell}) + T^j_{iq} \gamma_{kq}(\overline{e}_{\ell}) - T^q_{ik} \gamma_{qj}(\overline{e}_{\ell}) \right) \nonumber \\
& = & T^j_{ik , \overline{\ell}} - \sum_q \left( T^j_{qk} \overline{T^i_{q\ell }}  + T^j_{iq} \overline{T^k_{q\ell }} - T^q_{ik} \overline{T^q_{j\ell }} \right)\label{eq:deltaTbar},
\end{eqnarray}
where the index after the comma means covariant derivative with respect to $\nabla^b$. Then the equalities \eqref{eq:DR} and \eqref{eq:deltaTbar} imply that \[\sqrt{-1}\partial \overline{\partial} \omega = \frac{1}{4} \left\{ (T^{\ell}_{ik,\overline{j}} - T^j_{ik,\overline{\ell}})-P^{j\ell}_{ik} \right\} \varphi_i \wedge \varphi_k \wedge \overline{\varphi}_j \wedge \overline{\varphi}_{\ell}.\]
Therefore the equivalence in the proposition follows from the equality \eqref{dbT_refined}.
\end{proof}

\begin{proposition}\label{Tderivative}
Let $(M^n,g)$ be a Hermitian manifold, with $\omega$ its K\"ahler form.
Then it holds that
\begin{eqnarray}
\nabla^b_{1,0} T^c=0 & \Longleftrightarrow &
R^b_{ijk\bar{\ell}}=0\quad \Longrightarrow \quad \sum_r(T^r_{\ell i}T^k_{r j} + T^r_{j\ell}T^k_{ri} + T^r_{ij}T^k_{r \ell})=0,\label{parallel_10}\\
\nabla^b_{0,1} T^c=0 & \Longleftrightarrow &
R^b_{i \bar{j} k \bar{\ell}} - R^b_{k \bar{j} i \bar{\ell}} = -P_{ik}^{j \ell} \label{parallel_01},
\end{eqnarray}
where $\nabla^b_{1,0} T^c$ and $\nabla^b_{0,1} T^c$ are respectively the $(1,0)$- and $(0,1)$-components of $\nabla^bT^c$.
\end{proposition}

\begin{proof}
It is clear that the condition $\nabla^b_{1,0}T^c=0$ (that is, $T_{ij,k}^\ell=0$) implies $\sum_r(T^r_{\ell i}T^k_{r j} + T^r_{j\ell}T^k_{ri} + T^r_{ij}T^k_{r \ell})=0$ by \eqref{pT}.
Thus $R^b_{ijk\bar{\ell}}=0$ follows from \eqref{curv20}. Conversely, by \eqref{pT_refined} and \eqref{permutation}, $R^b_{ijk\bar{\ell}}=0$ implies that $T_{ij,k}^\ell=0$ and $\sum_r(T^r_{\ell i}T^k_{r j} + T^r_{j\ell}T^k_{ri} + T^r_{ij}T^k_{r \ell})=0$.
Therefore, the conclusion \eqref{parallel_10} is established.

In the meantime, by \eqref{dbT}, we know that the condition $\nabla^b_{0,1}T^c=0$ (that is, $T_{ik,\bar{\ell}}^j=0$) implies that
\[ R^b_{i \bar{\ell} k \bar{j}} - R^b_{k \bar{\ell} i \bar{j}} =P_{ik}^{j\ell}, \]
which yields
\[ R^b_{i \bar{j} k \bar{\ell}} - R^b_{k \bar{j} i \bar{\ell}} =P_{ik}^{\ell j}=-P_{ik}^{j \ell}.\]
Conversely, if $R^b_{i \bar{j} k \bar{\ell}} - R^b_{k \bar{j} i \bar{\ell}} =-P_{ik}^{j \ell}$, then by the fact that $P_{ik}^{j\ell}=-P_{ik}^{\ell j}$ we get
\[R^b_{i \bar{j} k \bar{\ell}} - R^b_{k \bar{j} i \bar{\ell}} = -(R^b_{i \bar{\ell} k \bar{j}} - R^b_{k \bar{\ell} i \bar{j}}),\]
so by \eqref{dbT_refined} we have
\[T_{ik,\bar{\ell}}^j = -\frac{1}{3}P_{ik}^{j \ell} - \frac{1}{3} (R^b_{i \bar{j} k \bar{\ell}} - R^b_{k \bar{j} i \bar{\ell}}) = -\frac{1}{3}P - \frac{1}{3}(-P) = 0.\]
This establishes the equivalence of \eqref{parallel_01}.
\end{proof}

\begin{proposition} \label{prop3.6}
Let $(M^n,g)$ be a BTP manifold. Then it holds that, for any $i,j,k,\ell$,
\begin{enumerate}
\item\label{R20=0} $R^b_{ijk\bar{\ell}}=0$,
\item\label{R11_12=34} $R^b_{i \bar{j} k \bar{\ell}} = R^b_{k \bar{\ell} i \bar{j}}$,
\item\label{Q_pll} $\nabla^b Q=0$,
\item\label{R=0} $R^b_{xy \chi w}=0$,
\end{enumerate}
where $x,y,w$ are any tangent vector and $\chi$ is the dual vector field of Gauduchon's torsion $1$-form $\eta$.
Also, under the BTP condition, we have
\begin{eqnarray}
&& \partial \overline{\partial} \omega =0 \quad \Longleftrightarrow \quad   P_{ik}^{j \ell}=0  \quad \Longleftrightarrow \quad
R^b_{i \bar{j} k \bar{\ell}} = R^b_{k \bar{j} i \bar{\ell}}, \label{eq_plcld}\\
&& \partial \eta=0,\quad \overline{\partial} \eta = -\sum_{i,j,k}\eta_k \overline{T^i_{jk}} \varphi_i \wedge \overline{\varphi}_j,
\end{eqnarray}
In particular, any BTP metric $g$ is always Gauduchon, namely, satisfies $\partial \overline{\partial} \omega^{n-1}=0$.
\end{proposition}

\begin{proof}
It follows from \eqref{eq:2.8} that the BTP condition is equivalent to $\nabla^b T^c=0$, which means $\nabla^b_{1,0} T^c=0$ and $\nabla^b_{0,1} T^c=0$. By \eqref{parallel_10} and \eqref{parallel_01}, we get \eqref{R20=0} and \eqref{Q_pll} immediately, as $P_{ik}^{j\ell}$ is $\nabla^b$-parallel. Since Gauduchon's torsion 1-form $\eta=\sum_k \eta_k \varphi_k = \sum_{i,k}T^i_{ik}\varphi_k$, its dual vector field $\chi=\sum_k \overline{\eta_k} e_k$ is clearly $\nabla^b$-parallel and thus
\[R^b_{xy\chi w} = \langle (\nabla^b_x\nabla^b_y-\nabla^b_y\nabla^b_x-\nabla^b_{[x,y]})\chi,w\rangle =0,\]
which yields \eqref{R=0}. As to the equality \eqref{R11_12=34}, from $R^b_{i \bar{j} k \bar{\ell}} - R^b_{k \bar{j} i \bar{\ell}} = -P_{ik}^{j \ell}$, it yields that
\[\begin{aligned}
R^b_{i \bar{j} k \bar{\ell}} - R^b_{k \bar{\ell} i \bar{j}} &=  R^b_{i \bar{j} k \bar{\ell}} - R^b_{k \bar{j} i \bar{\ell}} + R^b_{k \bar{j} i \bar{\ell}} - R^b_{k \bar{\ell} i \bar{j}} \\
&= -P_{ik}^{j \ell} + (\overline{R^b_{j \bar{k} \ell \bar{i}} - R^b_{\ell \bar{k} j \bar{i}}})\\
&= -P_{ik}^{j \ell} - \overline{P_{j \ell}^{k i}}\\
&= -P_{ik}^{j \ell} -P_{ki}^{j \ell} \\
&=0.
\end{aligned}\]
Under the condition $\nabla^bT^c=0$, the equivalence \eqref{eq_plcld} follows from Proposition \ref{plcld} and \eqref{parallel_01}. For any given point $p\in M$, let $e$ be a local unitary frame near $p$ such that $\theta^b|_p=0$. The structure equation gives us $d\varphi = - \,^t\!\theta \wedge \varphi + \tau =  \,^t\!\gamma \wedge \varphi + \tau$, hence $\partial \varphi = -\tau$ and $\overline{\partial} \varphi = -\overline{\gamma'}\wedge \varphi$ at $p$. Therefore
\[\begin{aligned}
\partial \eta &= \partial (\sum_i \eta_i \varphi_i)
= -\sum_{i,j}\eta_{i,j}\varphi_i \wedge \varphi_j - \frac{1}{2} \sum_{i,j,k} \eta_k T^k_{ij} \varphi_i \wedge \varphi_j,\\
\overline{\partial} \eta &=\overline{\partial} (\sum_i \eta_i \varphi_i)
= - \sum_{i,j} \eta_{i,\bar{j}} \varphi_i \wedge \overline{\varphi}_j -  \sum_{i,j,k} \eta_k \overline{T^i_{jk}} \varphi_i \wedge \overline{\varphi}_j.
\end{aligned}\]
Here the indices after comma stand for the covariant derivative with respect to $\nabla^{b}$. It is clear that $\nabla^b \eta=0$ and thus $\eta_{i,j}=0$, $\eta_{i,\bar{j}}=0$. Also, \eqref{parallel_10} implies that
\[  \sum_r(T^r_{\ell i}T^k_{r j} + T^r_{j\ell}T^k_{ri} + T^r_{ij}T^k_{r \ell})=0. \]
Let $k=\ell$ in the equation above and sum up, which yields that $\sum_r\eta_r T^r_{ij}=0$, and thus
\[\partial \eta=0,\quad \overline{\partial} \eta = -\sum_{i,j,k}\eta_k \overline{T^i_{jk}} \varphi_i \wedge \overline{\varphi}_j.\]
By the defining equation $\partial \omega^{n-1} = - \eta \wedge \omega^{n-1}$, we get
\[  \partial \overline{\partial} \omega^{n-1} =  (\overline{\partial} \eta
+  \eta \wedge \overline{\eta}) \wedge \omega^{n-1} = 0.\]
This completes the proof of the proposition.
\end{proof}

\vspace{0.4cm}
%\vspace{0.3cm}

\section{BTP manifolds and proof of Theorem \ref{theorem1.1}}\label{BTP}

In this section, we will prove Theorem \ref{theorem1.1}, which characterizes Hermitian manifolds with Bismut parallel torsion in terms of the behavior of the Bismut curvature $R^b$ alone.

First we will give the following two technical lemmata, the latter has already appeared in \cite[\S 4]{ZhaoZ19Str}, but we include the proof here for readers' convenience.

\begin{lemma}\label{swap}
Let $(M^n,g)$ be a Hermitian manifold. Then the followings are equivalent
\begin{enumerate}
\item $R^b_{i \bar{j} k \bar{\ell}} = R^b_{k \bar{\ell} i \bar{j}}$,
\item $Q_{i \bar{j} k \bar{\ell}}=-Q_{i\bar{\ell} k \bar{j}}= \overline{Q_{j\bar{i}\ell \bar{k}}}$,
\item $T^j_{ik ,\overline{\ell }} = - T^{\ell}_{ik ,\overline{j } } = \overline{  T^i_{j\ell  ,\overline{k }} }$,
\end{enumerate}
for any $i,j,k,\ell$.
\end{lemma}

\begin{proof}
Let us assume that $R^b_{i \bar{j} k \bar{\ell}} = R^b_{k \bar{\ell} i \bar{j}}$ holds. Then it follows that
\[\begin{aligned}
Q_{i\bar{j}k\bar{\ell}} &= R^b_{i\bar{j}k\bar{\ell}} - R^b_{k\bar{j}i\bar{\ell}}
= R^b_{k\bar{\ell}i\bar{j}} - R^b_{i\bar{\ell}k\bar{j}} = -Q_{i\bar{\ell}k\bar{j}},\\
\overline{Q_{i\bar{j}k\bar{\ell}}} &= R^b_{j\bar{i}\ell\bar{k}} - R^b_{j\bar{k}\ell\bar{i}}
= R^b_{j\bar{i}\ell\bar{k}} - R^b_{\ell\bar{i}j\bar{k}} = Q_{j\bar{i}\ell\bar{k}}.
\end{aligned}\]
By \eqref{dbT_refined}, we get
\[T^j_{ik,\bar{\ell}} = -\frac{1}{3}P_{ik}^{j \ell} - \frac{1}{3}Q_{i\bar{j}k\bar{\ell}},\]
which implies that
\[T^j_{ik ,\overline{\ell }} = - T^{\ell}_{ik ,\overline{j } } = \overline{  T^i_{j\ell  ,\overline{k }} }.\]
Conversely, assume that the equality $T^j_{ik ,\overline{\ell }} = - T^{\ell}_{ik ,\overline{j } } = \overline{  T^i_{j\ell  ,\overline{k }} }$ holds. Then by \eqref{dbT} we get
\[Q_{i\bar{\ell}k\bar{j}}=R^{b}_{i\bar{\ell}k\bar{j}}-R^b_{k\bar{\ell}i\bar{j}} =3 T_{ik,\bar{\ell}}^j + P^{j\ell}_{ik},\]
which implies that $Q_{i\bar{j}k\bar{\ell}} = -Q_{i\bar{\ell}k\bar{j}}$ and $\overline{Q_{i\bar{j}k\bar{\ell}}}= Q_{j\bar{i}\ell\bar{k}}$. Also, it holds that
\[\begin{aligned}
R^b_{i \bar{j} k \bar{\ell}} - R^b_{k \bar{\ell} i \bar{j}} &= R^b_{i \bar{j} k \bar{\ell}} - R^b_{k\bar{j}i\bar{\ell}}
+R^b_{k\bar{j}i\bar{\ell}} - R^b_{k \bar{\ell} i \bar{j}} \\
&= Q_{i\bar{j}k\bar{\ell}} + \overline{R^b_{j\bar{k}\ell\bar{i}} - R^b_{\ell \bar{k} j \bar{i}}}\\
&= Q_{i\bar{j}k\bar{\ell}} + \overline{Q_{j\bar{k}\ell\bar{i}}} \\
&= 0.
\end{aligned}\]
This completes the proof of the lemma. \end{proof}

\begin{lemma}\label{commt}
Let $(M^n,g)$ be a Hermitian manifold. If $R^b_{ijk\bar{\ell}}=0$ for any $i,j,k,\ell$, then the following commutation formulae hold
\begin{enumerate}
\item $\eta_{i, \overline{j}\,\overline{k} } - \eta_{i, \overline{k}\,\overline{j} }  =   \sum_r \overline{T^{r}_{kj}} \, \eta_{i, \overline{r}}$,
\item $T^i_{j\ell , \overline{k} \, \overline{p}} - T^i_{j\ell , \overline{p} \, \overline{k} } =   \sum_r \overline{T^r_{pk} } T^i_{j\ell , \overline{r}}$,
\end{enumerate}
for any $i,j,k,\ell,p$. Again indices after comma stand for covariant derivatives with respect to $\nabla^b$.
\end{lemma}

\begin{proof}
First let us prove the second equality. From the definition of covariant derivatives,
\[T_{j \ell,\bar{k}}^i= \overline{e}_k T_{j \ell}^i - \sum_r \{ T^i_{r\ell }\,\langle \nabla^b_{\overline{e}_k}e_j,\overline{e}_r \rangle  + T^i_{j r}\,\langle \nabla^b_{\overline{e}_k} e_{\ell},\overline{e}_r \rangle  - T^r_{j \ell}\,\langle \nabla^b_{\overline{e}_k}e_r, \overline{e}_i\rangle \} .\]
For $\forall\ x\in M$, choose a local unitary frame $e$ so that the matrix $\theta^b=0$ at the point $x$, then we have
\[T_{j \ell,\bar{k}\,\bar{p}}^i = \overline{e}_p\overline{e}_k T_{j \ell}^i -
\sum_r \{ T^i_{r\ell }\,\langle  \nabla^b_{\overline{e}_p}\nabla^b_{\overline{e}_k}e_j,\overline{e}_r \rangle  + T^i_{j r}\,\langle \nabla^b_{\overline{e}_p}\nabla^b_{\overline{e}_k} e_{\ell},\overline{e}_r\rangle   - T^r_{j \ell}\,\langle \nabla^b_{\overline{e}_p}\nabla^b_{\overline{e}_k}e_r,\overline{e}_i \rangle \} , \]
which implies that
\[ T^i_{j\ell , \overline{k} \, \overline{p}} - T^i_{j\ell , \overline{p} \, \overline{k} } =
[\overline{e}_p,\overline{e}_k] T_{j \ell}^i - \sum_r \{ T^i_{r\ell }\,R^b_{\bar{p}\bar{k}j\bar{r}} + T^i_{j r}\,R^b_{\bar{p} \bar{k} \ell \bar{r}} - T^r_{j \ell}\,R^b_{\bar{p}\bar{k}r\bar{i}} \} .\]
Clearly at the point $x$,
\[ [\overline{e}_p,\overline{e}_k ]= \nabla_{\overline{e}_p}\overline{e}_k-\nabla_{\overline{e}_k}\overline{e}_p = -\overline{T^r_{kp}} \overline{e}_r,\]
where $\nabla$ is the Levi-Civita connection. So the second equality in the lemma is established since $R^b_{ijk\bar{\ell}}=0$. By letting $i=j$ in the second equality and sum over, we get the first equality.
\end{proof}

Now we are ready to prove Theorem \ref{theorem1.1}.

\begin{proof}[{\bf Proof of Theorem \ref{theorem1.1}:}]
By Proposition \ref{prop3.6}, it is clear that the BTP condition implies each of the four conditions in Theorem \ref{theorem1.1},
so we just need to prove the converse, namely, the combination of conditions (\ref{eq:1.1})\,--\,(\ref{eq:1.4}) in Theorem \ref{theorem1.1} implies that the metric is BTP.
We will follow the strategy in the proof of \cite[Proposition 2]{ZhaoZ19Str} to show that $\nabla^bT^c=0$ holds and thus $\nabla^bT^b=0$  by \eqref{eq:2.8}.

Note that the condition (\ref{eq:1.1}) in Theorem \ref{theorem1.1} is equivalent to $\nabla^b_{1,0}T^c=0$ by \eqref{parallel_10}, which implies \begin{equation}\label{3T}
\sum_r(T^r_{\ell i}T^k_{r j} + T^r_{j\ell}T^k_{ri} + T^r_{ij}T^k_{r \ell})=0,
\end{equation}
and the fourth equality (\ref{eq:1.4}) means
\begin{equation}\label{vR}
\sum_{i,r} \overline{\eta_i} ( R^b_{i\bar{j}r\bar{r}} - R^b_{r \bar{j} i \bar{r}}) =0,   \ \ \  \forall \ j.
\end{equation}
In the meantime, By (\ref{eq:1.2}), we deduce that $\mathrm{Ric}(Q)$ is Hermitian symmetric:
\[\mathrm{Ric}(Q)_{i \bar{j}}= \sum_r R^b_{i \bar{j} r \bar{r}} - R^b_{r \bar{j} i \bar{r}}
= \overline{\sum_r R^b_{j \bar{i} r \bar{r}} - R^b_{j \bar{r} r \bar{i}}}
= \overline{\sum_r R^b_{j \bar{i} r \bar{r}} - R^b_{r \bar{i} j \bar{r}}}
=\overline{\mathrm{Ric}(Q)_{j \bar{i}}}.\]

Firstly, we will show that $|\eta|^2$ is a constant. We have $\eta_{i,j}=0$ since $R^b_{ijk\overline{\ell}}=0$.
By \eqref{dbeta} and \eqref{eq:1.2}, we get
\[|\eta|^2_{,\bar{j}}= \sum_i \eta_{i,\bar{j}}\overline{\eta_i}
=-\frac{1}{3}(\sum_{i,r}\eta_r \overline{T^i_{jr}} + \overline{\eta}_r T^j_{ir} - \sum_{i,t,r}T^j_{tr}\overline{T^i_{tr}})\overline{\eta_i} - \frac{1}{3}\sum_{i,r}(R^b_{i\bar{j}r\bar{r}} - R^b_{r\bar{j}i\bar{r}})\overline{\eta_i}=0\]
for any $j$. Here we used \eqref{vR} and the equality
\begin{equation}\label{eta_T}
\sum_r \eta_r T_{ij}^r =0,
\end{equation}
which follows from \eqref{3T} after we let $k=\ell$ and sum up. Therefore we know that $|\eta|^2$ is a constant.

Secondly, we will show that $\nabla^b \eta=0$. It suffices to prove $\nabla^b_{0,1}\eta=0$.
As $|\eta|^2$ is a constant, it yields that, after we take the covariant derivative in $\overline{\ell}$,
\[\sum_k \eta_{k,\bar{\ell}} \,\overline{\eta_k}=0.\]
Take the covariant derivative in $\ell$ and sum up $\ell$ again, then we get
\begin{equation}\label{eta_cst}
\sum_{k,\ell} \big( |\eta_{k,\bar{\ell}}|^2 + \eta_{k,\bar{\ell}\ell}\,\overline{\eta_{k}}  \big) =0.
\end{equation}
The equalities \eqref{dbeta} and \eqref{eq:1.2} imply that
\begin{equation}\label{dbeta_TPLL}
\eta_{i,\bar{j}}  =  -\frac{1}{3}( \overline{\phi_j^i} + \phi_i^j - B_{i\bar{j}})
- \frac{1}{3} \mathrm{Ric}(Q)_{i\bar{j}},
\end{equation}
thus $\eta_{i,\bar{j}}$ is Hermitian symmetric since $\mathrm{Ric}(Q)_{i\bar{j}}$ is so.
Then it yields
\[\begin{aligned}
 9 \sum_{k,\ell} |\eta_{k,\bar{\ell}}|^2 &= |\phi+\phi^*-B|^2 +  |\mathrm{Ric}(Q)|^2
+2 \mathrm{Re}\left(\phi \mathrm{Ric}(Q) + \overline{\phi \mathrm{Ric}(Q)} - B \mathrm{Ric}(Q)\right), \\
&= \left(|\phi + \phi^*|^2+|B|^2-4 \mathrm{Re}(\phi B)\right)
+  |\mathrm{Ric}(Q)|^2 + 4 \mathrm{Re} \left( \phi \mathrm{Ric}(Q) \right)- 2 \mathrm{Re} \left( B \mathrm{Ric}(Q) \right),
\end{aligned}\]
where
\[|\phi+\phi^*-B|^2 =\sum_{k,\ell}(\phi_k^\ell+\overline{\phi_\ell^k}-B_{k\bar{\ell}})\overline{(\phi_k^\ell+\overline{\phi_\ell^k}-B_{k\bar{\ell}})},
\quad |\mathrm{Ric}(Q)|^2=\sum_{k,\ell} \mathrm{Ric}(Q)_{k \bar{\ell}} \overline{ \mathrm{Ric}(Q)_{k \bar{\ell}}},\]
\[\phi\mathrm{Ric}(Q)=\sum_{k,\ell} \phi^k_\ell \mathrm{Ric}(Q)_{k \bar{\ell}}, \quad
B\mathrm{Ric}(Q) = \sum_{k,\ell} B_{k\bar{\ell}}\mathrm{Ric}(Q)_{\ell \bar{k}}, \quad \phi B = \sum_{k,\ell} \phi^k_{\ell} B_{k\bar{\ell}}.\]
Note that
\[|\phi + \phi^*|^2= \sum_{k,\ell}(\phi^\ell_k + \overline{ \phi^k_\ell} ) ( \overline{ \phi^\ell_k + \overline{\phi^k_\ell}} )
= 2|\phi|^2 + 2\mathrm{Re}( \phi \cdot \phi ),\]
where we wrote
\[ \phi \cdot \phi = \sum_{k, \ell} \phi_k^\ell \phi_\ell^k.\]
Then it follows from Lemma \ref{commt}, the $\nabla^b$-parallelness of $\mathrm{Ric}(Q)$, \eqref{dbeta_TPLL} and Lemma \ref{swap} that
\[\begin{aligned}
\sum_\ell \eta_{k,\bar{\ell} \ell} &= \sum_\ell \left( \overline{\eta_{\ell,\bar{k}}}\right)_{\!,\ell} = \sum_{\ell} \overline{\eta_{\ell,\bar{k}\,\bar{\ell}}}  = \sum_\ell  \big( \overline{\eta_{\ell, \bar{\ell}\,\bar{k}}
+  \sum_r \overline{T^r_{\ell k}}\eta_{\ell,\bar{r}}} \big)  \\
&= \overline{ \big(\sum_\ell \eta_{\ell,\bar{\ell}}\big)_{\!,\bar{k}}} +  \sum_{\ell,r} T^r_{\ell k}\eta_{r,\bar{\ell}} \\
&= \overline{\big\{ \frac{1}{3}(|T|^2-2|\eta|^2)-\frac{1}{3} \sum_r \mathrm{Ric}(Q)_{r \bar{r}}\big\}_{\!,\bar{k}}}
+  \sum_{\ell,r} T^r_{\ell k}\eta_{r,\bar{\ell}} \\
&=  \frac{1}{3}(|T|^2)_{\!,k}
+  \sum_{\ell,r} T^r_{\ell k}\eta_{r,\bar{\ell}}  =    \frac{1}{3} \sum_{i,j ,r} T^i_{j r} \overline{  T^i_{j r,\bar{k} }  } +   \sum_{\ell,r} T^r_{\ell k}\eta_{r,\bar{\ell}}  \\
&= \frac{1}{3} \sum_{i,j,r}T^i_{jr} \overline{T^i_{jr,\bar{k}}} + \frac{1}{3} \sum_{\ell ,r}T^r_{\ell k} \left( B_{r \bar{\ell}} - \phi_r^{\ell} - \overline{\phi_{\ell}^r}  - \mathrm{Ric}(Q)_{r \bar{\ell}} \right)
\end{aligned}\]
Hence, by \eqref{dbT_refined}, \eqref{eq:1.2} and \eqref{eta_T}, it yields
\[ \begin{aligned}
\sum_{k,\ell}\eta_{k,\bar{\ell}\ell} \,\overline{\eta_k} &=
\frac{1}{3} \sum_{i,j,k,r} \overline{\eta_k} T^i_{jr} T^j_{ik,\bar{r}} + \frac{1}{3} \sum_{k,\ell ,r} \overline{\eta_k} T^r_{\ell k} \left(  B_{r \bar{\ell}} - \phi_r^{\ell} - \overline{\phi_{\ell}^r} -  \mathrm{Ric}(Q)_{r \bar{\ell}} \right) \\
&= -\frac{1}{9} \sum_{i,j,k,r} \overline{\eta_k} T^i_{jr} \left(\sum_q T^q_{ik}\overline{T^q_{j r}} + T^j_{iq} \overline{T^k_{r q}}
- T^j_{kq}\overline{T^i_{r q}} - T^r_{iq} \overline{T^k_{jq}} + T^r_{kq}\overline{T^i_{jq}}
+ ( R^b_{i\bar{j}k\bar{r}} - R^b_{k \bar{j} i \bar{r}} )\right) \\
&\quad + \frac{1}{3}(\phi B - \phi \cdot \phi - |\phi|^2 - \phi \mathrm{Ric}(Q)) \\
&= -\frac{1}{9}(\phi B - 2\phi A)  - \frac{1}{9} \sum_{i,j,k,r} \overline{\eta_k} T^i_{jr} Q_{i\bar{j}k\bar{r}} + \frac{1}{3}( \phi B - \phi \cdot \phi - |\phi|^2  - \phi \mathrm{Ric}(Q))
 \\
&= \frac{2}{9}(\phi B + \phi A)- \frac{1}{3}(\phi \cdot \phi + |\phi|^2) - \frac{1}{3} \phi \mathrm{Ric}(Q)
- \frac{1}{9} \sum_{i,j,k,r} \overline{\eta_k} T^i_{jr} Q_{i\bar{j}k\bar{r}} ,
\end{aligned}\]
where
\[\phi A = \sum_{k,\ell}\phi^k_{\ell} A_{k \bar{\ell}}.\]
Plug the above into  the equality \eqref{eta_cst}, and multiply by 9, then we get
\[ |B- \phi - \phi^{\ast} - \mathrm{Ric}(Q) |^2 + 2 \mathrm{Re}(\phi B + \phi A) - 3\mathrm{Re}(\phi \cdot \phi + |\phi|^2) -3\mathrm{Re}(\phi \mathrm{Ric}(Q)) - \mathrm{Re}(\sum_{i,j,k,r} \overline{\eta_k} T^i_{jr} Q_{i\bar{j}k\bar{r}}) = 0. \]
That is
\begin{equation}\label{eq_1}
\begin{aligned}
|B|^2  + 2 \mathrm{Re}(\phi A - \phi B ) - 2 \mathrm{Re}(B\mathrm{Ric}(Q)) & \ = \
\frac{1}{2}|\phi+\phi^*|^2 -  \mathrm{Re}(\phi \mathrm{Ric}(Q))\\
& \ \ \ - |\mathrm{Ric}(Q)|^2 +  \mathrm{Re} \big(\sum_{i,j,k,r} \overline{\eta_k} T^i_{jr} Q_{i\bar{j}k\bar{r}}\big).
\end{aligned}
\end{equation}

On the other hand, by taking covariant derivative in $\bar{\ell}$
on both sides of the equality \eqref{eta_T}, we obtain
\[\sum_r \eta_{r,\bar{\ell}}\,T^r_{ij} + \eta_r\,T^r_{ij,\bar{\ell}} =0,\]
which means
\begin{gather}
\sum_r\left(\phi_r^{\ell} + \overline{\phi_{\ell}^r} - B_{r \bar{\ell}} + \mathrm{Ric}(Q)_{r\bar{\ell}} \right)T^r_{ij} \notag\\
+ \sum_r \eta_r \left(\sum_q (T^q_{ij}\overline{T^q_{r \ell}} + T^r_{iq} \overline{T^j_{\ell q}}
- T^r_{jq}\overline{T^i_{\ell q}} - T^\ell_{iq} \overline{T^j_{rq}} + T^\ell_{jq}\overline{T^i_{rq}})
+  ( R^b_{i\bar{r}j\bar{\ell}} - R^b_{j \bar{r} i \bar{\ell}} )\right)=0 , \notag
\end{gather}
or equivalently,
\begin{equation}\label{tr_etaT}
\sum_r\left((\phi_r^{\ell}-B_{r\bar{\ell}})T^r_{ij} +  \mathrm{Ric}(Q)_{r\bar{\ell}} T^r_{ij}
+  \overline{\phi^j_r}T^{\ell}_{ir} - \overline{\phi^i_r}T^{\ell}_{jr} +\eta_r Q_{i\bar{r}j\bar{\ell}}\right)=0.
\end{equation}
Multiply $\overline{\eta_j}$ on both sides of \eqref{tr_etaT} and sum up $j$, we obtain
\[ \sum_{r} \left(\phi_r^{\ell}\phi_i^{r} - \phi^r_i B_{r\bar{\ell}}
+  \phi^r_i \mathrm{Ric}(Q)_{r\bar{\ell}} +  \phi^{\ell}_{r}\overline{\phi^i_r}
+ \sum_j \eta_r \overline{\eta_j} Q_{i\bar{r}j\bar{\ell}}\right) =0,\]
where the equality \eqref{eta_T} is used again.
Let $i=\ell$ in the above identity and sum up $i$, and use \eqref{vR} to eliminate the last term, then we get
\[ \phi B = \phi \cdot \phi +  \phi \mathrm{Ric}(Q) + |\phi|^2,\]
hence
\begin{equation}\label{eq_2}
\mathrm{Re}(\phi B) = \mathrm{Re}(\phi \cdot \phi) +  \mathrm{Re}(\phi \mathrm{Ric}(Q)) + |\phi|^2
=\frac{1}{2}|\phi+\phi^*|^2 +  \mathrm{Re}(\phi \mathrm{Ric}(Q)).
\end{equation}
Similarly, multiply $\overline{T^{\ell}_{ij}}$ on both sides of \eqref{tr_etaT} and sum up $i,j,\ell$, it yields
\[ \phi B - |B|^2 +  B \mathrm{Ric}(Q) + 2 \overline{\phi A}
+ \sum_{i,j,r,\ell} \eta_r \overline{T^{\ell}_{ij}} Q_{i\bar{r}j\bar{\ell}}=0. \]
By taking the real part and utilizing Lemma \ref{swap}, we obtain
\begin{equation}\label{eq_3}
|B|^2 - \mathrm{Re}(\phi B) - 2 \mathrm{Re}(\overline{\phi A}) -  \mathrm{Re}( B \mathrm{Ric}(Q))
=-  \mathrm{Re} \left( \sum_{i,j,r,\ell} \overline{\eta_r} T^{\ell}_{ij}Q_{\ell \bar{i} r \bar{j}} \right).
\end{equation}
Note that the equality \eqref{vR} meams that
\[\sum_{i}\overline{\eta_i}\mathrm{Ric}(Q)_{i\bar{j}}=0, \ \ \ \ \forall \ j.\]
By taking covariant derivative in $k$ and using the assumption that  $\nabla^b \mathrm{Ric}(Q)=0$, we get
\[\sum_i \eta_{i,\bar{k}} \mathrm{Ric}(Q)_{j \bar{i}}=0,\]
hence
\[ \sum_i \left( \phi_i^k + \overline{\phi_k^i} - B_{i\bar{k}} +  \mathrm{Ric}(Q)_{i\bar{k}}\right) \mathrm{Ric}(Q)_{j \bar{i}}=0. \]
Let $j=k$ and sum up $j$, so we have
\[ \phi \mathrm{Ric}(Q) + \overline{\phi \mathrm{Ric}(Q)} - B \mathrm{Ric}(Q) + |\mathrm{Ric}(Q)|^2 =0, \]
and thus
\begin{equation}\label{eq_4}
\mathrm{Re}(B\mathrm{Ric}(Q)) = 2 \mathrm{Re}(\phi \mathrm{Ric}(Q))  + |\mathrm{Ric}(Q)|^2 .
\end{equation}
In the meantime, by multiplying $\overline{T^k_{\ell j}}$ on both sides of \eqref{3T} and summing up $j,k,\ell$, we obtain
\[ \sum_{r,\ell} T^r_{\ell i}A_{r \bar{\ell}} - \sum_{k,r} T^k_{r i}B_{k \bar{r}} +  \sum_{j,r}T^r_{ji}A_{r\bar{j}}=0.\]
Hence
\begin{equation}\label{TBTA}
2 \sum_{r,\ell}T^r_{\ell i}A_{r \bar{\ell}} = \sum_{r,\ell} T^r_{\ell i}B_{r \bar{\ell}}.
\end{equation}
Multiply $\overline{\eta_i}$ on both sides above and sum up $i$, which yields
\[2 \phi A = \phi B,\]
and thus
\begin{equation}\label{eq_5}
2 \mathrm{Re} (\phi A) = \mathrm{Re} (\phi B).
\end{equation}
If we take the difference between \eqref{eq_1} and \eqref{eq_3}, and plug in  \eqref{eq_2}, \eqref{eq_4} and \eqref{eq_5}, then we get
\[ \mathrm{Re} \left(\sum_{i,j,k,r} \overline{\eta_k} T^i_{jr} Q_{i\bar{j}k\bar{r}}\right)=0.\]
Therefore, the equalities \eqref{eq_2} and \eqref{eq_4}  give us
\begin{eqnarray*}
&& \mathrm{Re}(\phi B)= 2\mathrm{Re}(\phi A)= \frac{1}{2}|\phi+\phi^*|^2 +  \mathrm{Re}(\phi  \mathrm{Ric}(Q)),\\
&& \mathrm{Re}(B\mathrm{Ric}(Q)) = 2 \mathrm{Re}(\phi \mathrm{Ric}(Q))  + |\mathrm{Ric}(Q)|^2,
\end{eqnarray*}
while \eqref{eq_3}, together with the above and \eqref{eq_5} give us
\[|B|^2 = |\phi + \phi^*|^2 + 4\mathrm{Re}(\phi \mathrm{Ric}(Q)) + |\mathrm{Ric}(Q)|^2.\]
Putting all these together, we compute that
\begin{eqnarray*}
9\sum_{k,\ell}|\eta_{k,\bar{\ell}}|^2   & = & |B-\phi -\phi^{\ast} - \mathrm{Ric}(Q)|^2 \\
& = & |\phi +\phi^{\ast}|^2 -4 \mathrm{Re}(\phi B) + |B|^2 -2 \mathrm{Re}(B\mathrm{Ric}(Q))  +4\mathrm{Re}(\phi \mathrm{Ric}(Q))  + |\mathrm{Ric}(Q)|^2\\
& = & |\phi +\phi^{\ast}|^2 -4 \{ \frac{1}{2}|\phi+\phi^*|^2 +  \mathrm{Re}(\phi  \mathrm{Ric}(Q)) \} + \{ |\phi + \phi^*|^2 + 4\mathrm{Re}(\phi \mathrm{Ric}(Q)) + |\mathrm{Ric}(Q)|^2\} \\
& & - 2\{ 2 \mathrm{Re}(\phi \mathrm{Ric}(Q))  + |\mathrm{Ric}(Q)|^2 \}  + 4\mathrm{Re}(\phi \mathrm{Ric}(Q))  + |\mathrm{Ric}(Q)|^2\\
& = & 0,
\end{eqnarray*}
so we must have $\eta_{k,\bar{\ell}} =0$ for any $k,\ell$, that is,  $\nabla^b\eta=0$.

Thirdly, we will show that $|T|^2$, $\phi_i^j$ and $B_{i \bar{j}}$ are $\nabla^b$-parallel. To this end, note that the equality \eqref{dbeta_TPLL} now means $$ \mathrm{Ric}(Q) = B - \phi - \phi^{\ast}. $$
Taking trace, we get
\[  \sum_i \mathrm{Ric}(Q)_{i\bar{i}}  = |T|^2-2|\eta|^2.\]
Since $|\eta|^2$ is a constant and $\nabla^b \mathrm{Ric}(Q)=0$, we know that $\nabla^b |T|^2 =0$, hence $|T|^2$ is a constant. It is clear that $\phi_{i,k}^j=0$ as $\nabla^b_{1,0}T^c=0$ and $\nabla^b \eta=0$, while we have
\[ \phi_{i,\bar{k}}^j = \sum_r \overline{\eta_r} T^j_{ir,\bar{k}} =- \sum_r \overline{\eta_r} T^j_{ri,\bar{k}}
= -\sum_r \overline{\eta_r T^r_{jk,\bar{i}}}=0,\]
where Lemma \ref{swap} and the equality \eqref{eta_T} are used. Hence $\phi_i^j$ is $\nabla^b$-parallel. By \eqref{dbeta_TPLL} which now says that $B=\mathrm{Ric}(Q) +\phi + \phi^{\ast} $, we know that $B_{i \bar{j}}$ is also $\nabla^b$-parallel.

Finally, we are ready to show that $\nabla^b T^c=0$.
It follows from \eqref{TBTA} that
\begin{equation}\label{vTA}
 2\sum_{i,\ell,r} \left(T^r_{\ell i} A_{r \bar{\ell}}\right)_{,\bar{i}}
= \sum_{i,\ell,r} \left(T^r_{\ell i}B_{r \bar{\ell}}\right)_{,\bar{i}}
= \sum_{i,\ell,r} T^r_{\ell i,\bar{i}}B_{r \bar{\ell}}
= - \sum_{i,\ell,r} T^i_{\ell i,\bar{r}}B_{r \bar{\ell}}
=  \sum_{\ell,r} \eta_{\ell,\bar{r}}B_{r \bar{\ell}}=0,
\end{equation}
where Lemma \ref{swap} is used.
Then, since $|T|^2$ is $\nabla^b$-parallel, it follows that, after we take the covariant derivative in $\bar{\ell}$,
\[0=|T|^2_{,\bar{\ell}}=\sum_{i,j,k}\overline{T^j_{ik}}T^j_{ik,\bar{\ell}} =0.\]
Take the covariant derivative in $\ell$ again and sum up $\ell$, so we get
\[\sum_{i,j,k,\ell} |T^j_{ik,\bar{\ell}}|^2 + \overline{T^j_{ik}}T^j_{ik,\bar{\ell}\,\ell} =0.\]
Note that, from Lemma \ref{swap}, Lemma \ref{commt} and $\nabla^b \eta=0$,
\[\begin{aligned}
\sum_{i,j,k,\ell} \overline{T^j_{ik}}T^j_{ik,\bar{\ell}\,\ell}& =
\sum_{i,j,k,\ell} \overline{T^j_{ik}} \left(\overline{T_{j \ell,\bar{k}}^i}\right)_{,\ell}
= \sum_{i,j,k,\ell} \overline{T^j_{ik} T_{j \ell,\bar{k}\,\bar{\ell}}^i}
= \sum_{i,j,k,\ell} \overline{T^j_{ik}( T_{j \ell,\bar{\ell}\,\bar{k}}^i +  \sum_r \overline{T^r_{\ell k}} T^i_{j \ell, \bar{r}} )} \\
&= - \sum_{i,j,k,\ell} \overline{T^j_{ik} T^{\ell}_{j \ell, \bar{i}\,\bar{k}}}
+  \sum_{i,j,k,\ell,r} T^r_{\ell k} \overline{T^j_{ik}} T^j_{ir,\bar{\ell}}
= \sum_{i,j,k} \overline{T^j_{ik}\eta_{j,\bar{i}\,\bar{k}}}
+  \sum_{k,\ell,r} T^r_{\ell k}A_{r\bar{k},\bar{\ell}} \\
&= - \sum_{k,\ell,r} \left( T^r_{k \ell}A_{r\bar{k}} \right)_{,\bar{\ell}}
+ \sum_{k,\ell,r} T^r_{k \ell, \bar{\ell}} A_{r\bar{k}}
= - \sum_{k,\ell,r} \left( T^r_{k \ell}A_{r\bar{k}} \right)_{,\bar{\ell}}
- \sum_{k,\ell,r} T_{k \ell,\bar{r}}^\ell A_{rk} \\
&= - \sum_{k,\ell,r} \left( T^r_{k \ell}A_{r\bar{k}} \right)_{,\bar{\ell}}
+ \sum_{k,r} \eta_{k,\bar{r}}A_{r\bar{k}}
=-2 \sum_{k,\ell,r} \left( T^r_{k \ell}A_{r\bar{k}} \right)_{,\bar{\ell}} \\
&=0,
\end{aligned}\]
where the last equality is due to \eqref{vTA}. Therefore, $T^j_{ik,\bar{\ell}}=0$ and the proof is completed.
\end{proof}

\vspace{0.4cm}

\section{Properties of BTP manifolds}\label{BTP_mfd}

In this section we will discuss the examples and properties of BTP manifolds and prove Propositions \ref{prop1.5}\,--\,\ref{cor1.11}.
We start with the proof of Proposition \ref{prop1.5} which gives a number of general properties for BTP manifolds.

\begin{proof}[{\bf Proof of Proposition \ref{prop1.5}.}]
Let $(M^n,g)$ be a BTP manifold. Note that the equalities (\ref{eq:1.5}) and (\ref{eq:1.7})\,--\,(\ref{eq:1.9})
have already been established in Lemma \ref{lemma3.1}, Propositions \ref{Tderivative} and \ref{prop3.6}.
Multiple the equality (\ref{eq:1.5}) by $\overline{T_{jk}^\ell}$ and sum up $j,k,\ell$, then the equality (\ref{eq:1.6}) is also established.
So we will start with the proof of (\ref{eq:1.10}).

Fix any $p\in M$. Choose a local unitary frame $e$ in a neighborhood of $p$
so that the matrix of Bismut connection vanishes at $p$, i.e. $\theta^b|_p=0$.
Denote by $\varphi$ the coframe dual to $e$. Then the Chern connection matrix satisfies $\theta|_p=-\gamma|_p$.
Hence by the structure equation
$ d\varphi = - \,^t\!\theta \wedge \varphi +\tau$, we have
$$ \partial \varphi_i = -\frac{1}{2}\sum_{r,s}T^i_{rs}\varphi_r \wedge \varphi_s, \ \ \ \overline{\partial}\varphi_j = - \sum_{t,r} \overline{T^t_{jr}} \, \overline{\varphi}_r \wedge \varphi_t, \ \ \ \ \ \mbox{at} \ \ p. $$
The conjugation of the second equality of (\ref{eq:1.8}) implies that $\partial \overline{\eta} = \sum_{i,j} \phi^j_i \varphi_i \wedge \overline{\varphi}_j$. After taking $\partial $ again and using the fact that $\nabla^b\phi =0$, we know that at $p$
\begin{eqnarray*}
 0 & = &  \partial^2\overline{\eta} = \sum_{i,j} \phi^j_i \big\{ -\frac{1}{2}\sum_{r,s} T^i_{rs} \varphi_r\varphi_s \wedge \overline{\varphi}_j - \varphi_i \wedge  \sum_{t,r} T^t_{jr} \varphi_r \overline{\varphi}_t \big\} \\
 & = & \frac{1}{2} \sum_{i,j,k,r} \big( - \phi^j_r T^r_{ik} + \phi^r_i T^j_{rk} + \phi^r_k T^j_{ir} \big) \,\varphi_i\varphi_k\overline{\varphi}_j .
\end{eqnarray*}
Thus (\ref{eq:1.10}) is established. Next we will prove (\ref{eq:1.11}).
Let us multiply $\overline{\eta}_k$ onto both sides of the equality (\ref{eq:1.7}) and sum up $k$.
The equality (4) of Proposition \ref{prop3.6} indicates the left hand side becomes zero, so we end up with
\begin{equation} \label{eq:4.1}
 0 = \sum_r \{ -\phi^j_r \overline{T^i_{\ell r}} + \phi^{\ell }_r \overline{T^i_{j r}}  - \phi^r_i \overline{T^r_{j\ell }}  \} , \ \ \ \ \ \forall \ i,j,\ell.
 \end{equation}
Multiply by $\eta_{\ell}$ onto the above equation and sum up $\ell$.
By utilizing the fact that $\sum_{\ell} T^{\ell}_{rs}\eta_{\ell} =0$ which can be derived from (\ref{eq:1.5}), we get
$$ 0 = \sum_r \{ \phi^j_r \overline{\phi^i_r } - \phi^r_i \overline{\phi^r_j } \} .$$
that is, $\phi$ commutes with $\phi^{\ast}$. If we multiply $T^s_{j\ell }$ onto (\ref{eq:4.1}) and sum up $j$, $\ell$,
then we have \[\sum_r \phi_i^r B_{r\bar{s}} = 2 \sum_{j,\ell,r} \phi^j_r T^s_{j \ell} \overline{T^i_{r \ell}}
=\sum_{r} B_{i\bar{r}} \phi_r^s.\]
Similarly, if we multiply $T^i_{\ell s}$ onto (\ref{eq:4.1}) and sum up $i$, $\ell$, then we have
\[\sum_r A_{s\bar{r}}\phi_r^j
=\sum_{i,\ell,r}(\phi^{\ell}_r T^i_{\ell s} \overline{T^i_{jr}} + \phi_i^r T^i_{s \ell}\overline{T^r_{j \ell}})
=\sum_r \phi_s^r A_{r \bar{j}}.\]
Note that here we have used (\ref{eq:1.5}) in obtaining  both of the above equalities. Hence, $\phi$ commutes with both $A$ and $B$.

Finally we will show that $A$ commutes with $B$.
To this end, let us choose our local unitary frame $e$ so that $A$ is diagonal, which can be decomposed into
several blocks, where diagonal entries remain the same in each block, while different blocks enjoy distinct representing entries.
Since $\nabla^b A=0$, it yields that
$$ d A_{i\bar{j}} = \sum_r \{  A_{r\bar{j}} \theta^b_{ir} - A_{i\bar{r}} \theta^b_{rj} \} , \ \ \ \ \ \ \forall \  i,j. $$
This implies that $\theta^b_{ij}=0$ whenever $A_{i\bar{i}} \neq A_{j\bar{j}}$,
so $\theta^b$ is block-diagonal with respect to those blocks of $A$.
By the structure equation $\Theta^b=d\theta^b - \theta^b\wedge \theta^b$,
the matrix of Bismut curvature $\Theta^b$ is also block-diagonal.
This means that $R^b_{i\bar{j}\ast \bar{\ast}} = R^b_{\ast \bar{\ast}i\bar{j}} =0$
whenever $A_{i\bar{i}} \neq A_{j\bar{j}}$, i.e. $i$ and $j$ are from different blocks.
In the meantime, for any $i$ and $j$ from different blocks,
$$ \mathrm{Ric}(Q)_{i \bar{j}} = \sum_r \big\{ R^b_{i\bar{j}r\bar{r}} - R^b_{r\bar{j}i\bar{r}} \big\} = - \sum_r R^b_{r\bar{j}i\bar{r}} =0,$$
since $r$ cannot belong to the same block with $i$ and $j$ simultaneously.
This shows that $\mathrm{Ric}(Q)$ is block-diagonal with respect to those $A$-blocks, hence it commutes with $A$.
On the other hand, by letting $k=\ell$ and summing up in (\ref{eq:1.7}), we get $\mathrm{Ric}(Q)=B-\phi -\phi^{\ast}$.
Since we already showed that $A$ commutes with $\phi$ hence $\phi^{\ast}$, we know that $A$ commutes with $B$ as well.
This completes the proof of Proposition \ref{prop1.5}.
\end{proof}

We will then discuss Proposition \ref{prop1.7} about the existence of admissible frames and
other properties for non-balanced BTP manifolds. Let $(M^n,g)$ be a non-balanced BTP manifold.
Denote by $\chi$ the vector field of type $(1,0)$  dual to $\eta$, which is characterized by
$\langle Y, \overline{\chi } \rangle =\eta (Y)$ for any type $(1,0)$-vector field $Y$.

\begin{proof}[{\bf Proof of Proposition \ref{prop1.7}.}]
Let $(M^n,g)$ be a non-balanced BTP manifold. It is clear that $|\eta|^2$ is a positive constant,
which will be denoted by $\lambda^2$. Thus we may choose the unitary coframe $\varphi$ dual to $e$
such that $\eta = \lambda \varphi_n$, which means that $\eta_1=\cdots=\eta_{n-1}=0$ and $\eta_n=\lambda$, thus $\chi=\lambda e_n$.
Then the first equality in \eqref{eq:1.12} indicates that $T^n_{ij}=0$ for any $i,j$,
which implies $\phi_i^j=\lambda T^j_{in}$ and $\phi_n^{j} = \phi^n_{j}=0$.
By (\ref{eq:1.11}), we know that the matrix $\phi = (\phi_i^j)$ commutes with $\phi^{\ast}$ hence is normal,
where $(\phi^{\ast})_i^j=\phi_j^i$.  So around each given point,
by a unitary change of $\{ \varphi_1 , \ldots , \varphi_{n-1}\}$ with $\varphi_n$ fixed,
we could always make $\phi$ diagonal, that is, $T^j_{in}=a_i\delta_{ij}$, where $a_n=0$ and $a_1+\cdots +a_{n-1}=\lambda$.
This gives us an admissible local unitary frame.

Then, under such admissible frame $e$, the vector $e_n=\frac{1}{\lambda}\chi$ is parallel under $\nabla^b$, so the Bismut connection matrix $\theta^b$ satisfies $\theta^b_{n\ast}=\theta^b_{\ast n}=0$. It is easy to see that, for any type $(1,0)$ vector field $X$,
\[\nabla^c_{\overline{X}}\chi=\lambda (\nabla^b_{\overline{X}} -  \gamma_{\overline{X}})e_n=\lambda \sum_r \big( \theta^b_{nr}(\overline{X}) + \overline{T^n_{rX}} \big)e_r=0,\]
thus $\chi$ is holomorphic. By the structure equation $\Theta^b=d\theta^b-\theta^b\wedge \theta^b$, the curvature matrix $\Theta^b$ under $e$ will also be block-diagonal, taking the form
$$ \Theta^b = \left[ \begin{array}{cc} \ast & 0 \\ 0 & 0 \end{array} \right]\!\!, $$
where $\ast$ is an $(n\!-\!1)\times (n\!-\!1)$ matrix. As is well-known,
the Lie algebra of the restricted holonomy group $\mbox{Hol}_{0}(\nabla^b)$ of $\nabla^b$
is generated by conjugation classes of $R^b_{xy}=\Theta^b(x,y)$ under $\nabla^b$-parallel transports,
and all of them operate in the $\nabla^b$-parallel block $e_n^{\perp}$, so $\mbox{Hol}^{0}(\nabla^b) \subseteq U(n-1) \times 1$.
Henece, we have established the existence of admissible frames and the reduction of Bismut holonomy.

It remains to show that $g$ is locally conformally balanced
if and only if all $a_i$ are real, and also when $M$ is compact, it can never be globally conformally balanced.

Let $e$ be an admissible frame. Then by (\ref{eq:1.8}) we have $\partial \eta =0$ and $\overline{\partial}\eta = -\lambda \sum_i \overline{a}_i \varphi_i \wedge \overline{\varphi}_i $, which implies
$$ d(\eta + \overline{\eta}) = \overline{\partial}\eta + \partial \overline{\eta}
= \lambda \sum_i (a_i - \overline{a}_i) \varphi_i \wedge \overline{\varphi}_i.$$
So the metric $g$ will be locally conformally balanced, i.e. the $1$-form $\eta + \overline{\eta} $ is closed,
if and only if all $a_i$ are real. If $g$ is globally conformally balanced, namely,
$\eta +\overline{\eta} = df$ for some real smooth function $f$ on $M$, then when $M$ is compact, we have
$$ \int_M \eta \wedge \overline{\eta} \wedge \omega^{n-1} = - \int_M \overline{\partial}\eta \wedge \omega^{n-1} = -\int_M \overline{\partial}\partial f \wedge \omega^{n-1} = \int_M f \partial\overline{\partial} (\omega^{n-1}) =0,$$
thus $\eta =0$, which contradicts with our assumption that $g$ is not balanced.
Here we use the fact that a BTP metric is necessarily Gauduchon, which was shown in \eqref{eq:1.13}.
Therefore when $M$ is compact, the non-balanced BTP metric $g$ can never be globally conformally balanced.
This completes the proof of Proposition \ref{prop1.7}
\end{proof}

\begin{remark}
It is clear that, for non-balanced BTP manifolds, the distribution generated by $\chi$ is a foliation,
as it is easy to verify $[\chi,\overline{\chi}]=0$ under the admissible frames.
Then we will also call it the canonical foliation as in the theory of Vaisman manifolds $($cf. \cite{Ts94,Ts97,OV22,Ist}$)$.
\end{remark}

The following notation of degenerate torsion, introduced in \cite[Defintion 5]{YZZ}, could also be applied to non-balanced BTP manifolds.

\begin{definition}\label{degtor_def}
A non-balanced BTP manifold $(M^n,g)$ is said to have {\em degenerate torsion,}
if under any admissible frame $e$, $T^{\ast }_{ik}=0$ for any $i,k <n$.
\end{definition}

\begin{definition}[Belgun \cite{B12}]\label{LP_GCE}
A Hermitian manifold $(M^n,g)$ is {\em LP} (which stands for `Lee potential') if the Gauduchon torsion $1$-form $\eta$ satisfies
\[\partial \eta =0 , \quad \partial \omega = c \,\eta \wedge\partial \overline{\eta},\]
where $c$ is a constant. A Hermitian manifold is {\em GCE}  (which stands for `generalized Calabi-Eckmann') if it is LP and BTP.
As shown in \cite[Proposition 3.2]{B12}, both the standard Hermitian structure and the modified ones on the Sasakian product are non-K\"ahler GCE manifolds.
\end{definition}

\begin{proposition}\label{BTP_LP}
Let $(M^n,g)$ be a non-balanced BTP manifold for $n\geq 3$.
Then the metric $g$ satisfies the LP condition if and only if it has degenerate torsion.
In particular, when $n=3$, non-balanced BTP threefolds coincide with non-K\"ahler GCE threefolds
in the sense of Belgun.
\end{proposition}

\begin{proof}
Let $e$ be an admissible frame on the non-balanced BTP manifold.
From Proposition \ref{prop1.5}, we have
\begin{gather*}
\eta \wedge \partial \overline{\eta}\ =\  \lambda^2 \varphi_n \wedge \sum_{i,j} T_{i n}^j \varphi_i \wedge \overline{\varphi_j}
\ = \ -\lambda^2 \sum_{i < n} T_{i n}^i \varphi_i \wedge\varphi_n \wedge \overline{\varphi_i},\\
-\sqrt{-1}\partial \omega\ =\ ^t\!\tau \wedge \overline{\varphi}
\ =\ \frac{1}{2}\sum_{i,j,k} T_{ik}^j \varphi_i \wedge \varphi_k \wedge \overline{\varphi_j}
\ = \ \sum_{\begin{subarray}{c}i<k\\j<n\end{subarray}}T^j_{ik}\varphi_i \wedge \varphi_k \wedge \overline{\varphi_j}.
\end{gather*}
So the condition $\partial \omega = c \,\eta \wedge \partial \overline{\eta}$ implies $T^{j }_{ik}=0$ for any $i,k <n$.
Conversely, the torsion degeneracy condition $T^{j }_{ik}=0$ for any $i,k <n$ yields
\[-\sqrt{-1}\partial \omega =\sum_{\begin{subarray}{c}i<n\\j<n\end{subarray}} T^j_{in}\varphi_i \wedge \varphi_n \wedge \overline{\varphi_j}
=\sum_{i<n} T^i_{in}\varphi_i \wedge \varphi_n \wedge \overline{\varphi_i}, \]
thus $\partial \omega = c \,\eta \wedge \partial \overline{\eta}$,
which shows that the LP condition is equivalent to the torsion degeneracy condition.

As to the statement that the set of non-balanced BTP threefolds is equal to
the set of non-K\"ahler GCE threefolds when $n=3$, it suffices to show
that any non-balanced BTP threefold $(M^3,g)$ always has degenerate torsion.
To see this, fix a local admissible frame $e$ on $M^3$. It follows that
$\eta_1=\eta_2=0$, $\eta_3=\lambda >0$, and $T^3_{ij}=0$ for $\forall\ i,j$.
Then it yields that
$$ T^1_{12} = \eta_2- T^3_{32} = 0, \ \ \ T^2_{12} = - T^2_{21} = -\eta_1 + T^3_{31} = 0.$$
That is, $T^{\ast}_{12}=0$, so the manifold has degenerate torsion. This completes the proof of Proposition \ref{BTP_LP}.
\end{proof}

\begin{remark}
This proposition actually implies that a Hermitian manifold which is both BKL and GCE is a BKL manifold with degenerate torsion. This was shown in \cite[Theorem 9]{YZZ}.
\end{remark}

Let us prove Propositions \ref{prop1.8} and \ref{prop1.9}. %The former is due to Andrada and Villacampa \cite{AndV}, which states that a locally conformally K\"ahler metric $g$ is BTP if and only if it is Vaisman. We will include a proof here for readers' convenience. The latter says that in each conformal class of Hermitian metrics, there is at most one BTP metric (up to constant multiples) unless the metric is locally conformally K\"ahler. Note that when the manifold is compact, this is clearly true since BTP metrics are Gauduchon.

\begin{proof}[{\bf Proof of Proposition \ref{prop1.8}.}]
The first statement is due to Andrada and Villacampa \cite{AndV}. Let us give a proof here for readers'convenience and also to benefit our later discussion.
For $\forall\ p\in M$, choose the local unitary frame $e$
so that the matrix $\theta^b$ of Bismut connection vanishes at $p$.
Denote by $\varphi$ the local unitary coframe dual to $e$.
The Levi-Civita connection $\nabla$ has connection matrix
$$ \nabla e = \theta^1 e + \overline{\theta^2} \,\overline{e}, \ \ \ \theta^1=\theta^b -\frac{1}{2}\gamma , \ \ \
\gamma_{ij} = \sum_k \big( T^j_{ik}\varphi_k - \overline{T_{jk}^i } \overline{\varphi_k} \big) , \ \ \ \theta^2_{ij} = \frac{1}{2}\sum_k \overline{T^k_{ij}} \varphi_k.$$
Since $\theta^b|_p=0$, or equivalently $\theta\big|_p = -\gamma\big|_p$, from the structure equation, we get
$$ \frac{1}{\sqrt{-1}} d \omega = d(\,^t\!\varphi \overline{\varphi }) = \,^t\!(\,^t\!\gamma \varphi +\tau )\overline{\varphi }
- \,^t\!\varphi  \overline{( \,^t\!\gamma \varphi +\tau  ) }=  \,^t\!\tau \overline{\varphi } - \,^t\!\varphi  \overline{\tau  }, $$
as $\gamma$ is skew-Hermitian. Thus we always have $\partial \omega = \sqrt{-1} \,^t\!\tau \wedge \overline{\varphi }$.
Let $g$ be a locally conformally K\"ahler metric,
i.e. there is a closed $1$-form $\psi$ such that $d\omega = \psi \wedge \omega$,
where this $\psi$ is necessarily equal to $-\frac{1}{n-1}(\eta +\overline{\eta})$,
so $\partial \omega = -\frac{1}{n-1}\eta \wedge \omega$. This implies that $\tau = -\frac{1}{n-1}\eta \wedge \varphi$,
or equivalently,
\begin{equation}\label{eq:4.11}
T^j_{ik} = \frac{1}{n-1}(\eta_k \delta_{ij} - \eta_i \delta_{kj}), \ \ \ \ \ \forall \ 1\leq i,j,k\leq n.
\end{equation}
%Gauduchon's torsion $1$-form $\eta$ is locally equal to $\eta = (n-1)\partial u$ for some smooth real-valued local function $u$.
Also by $\theta^b|_p=0$, we have $\theta^1\big|_p=-\frac{1}{2}\gamma\big|_p$, thus
\begin{eqnarray*}
\nabla \varphi_i & = & - \theta^1_{ji} \otimes \varphi_j - \theta^2_{ji} \otimes \overline{\varphi}_j \ \ \ = \ \ \  \frac{1}{2}\gamma_{ji} \otimes \varphi_j - \theta^2_{ji} \otimes \overline{\varphi}_j \\
& = & \frac{1}{2}\sum_{j,k} \big( T^i_{jk} \varphi_k\otimes \varphi_j - \overline{T^j_{ik} } \overline{\varphi}_k \otimes \varphi_j -  \overline{T^k_{ji} } \varphi_k \otimes \overline{\varphi}_j  \big) \\
& = & \frac{1}{2(n-1)}\big\{ \eta \otimes \varphi_i - \varphi_i \otimes \eta + \varphi_i \otimes \overline{\eta } - \overline{\eta } \otimes \varphi_i + \overline{\eta}_i \sum_k \big( \overline{\varphi}_k \otimes \varphi_k - \varphi_k \otimes \overline{\varphi}_k \big) \big\} \\
& = & \frac{1}{n-1} \big\{ \eta \wedge \varphi_i + \varphi_i \wedge \overline{\eta} - \overline{\eta}_i \sum_k  \varphi_k\wedge \overline{\varphi}_k \big\}.
\end{eqnarray*}
Therefore
\begin{eqnarray*}
\nabla \eta & = & \nabla \sum_i \eta_i \varphi_i \  \, = \ \, \sum_{i,j} \big( \eta_{i,j} \varphi_j \otimes \varphi_i + \varphi_{i,\overline{j}} \overline{\varphi}_j \otimes \varphi_i \big) + \sum_i \eta_i \nabla \varphi_i \\
 & = & \sum_{i,j} \big( \eta_{i,j} \,\varphi_j \otimes \varphi_i + \eta_{i,\overline{j}} \,\overline{\varphi}_j \otimes \varphi_i \big) + \frac{1}{n-1} \big\{ \eta \wedge  \overline{\eta} - |\eta|^2  \sum_k  \varphi_k\wedge \overline{\varphi}_k \big\} ,
\end{eqnarray*}
where indices after comma denote covariant derivatives with respect to $\nabla^b$.

If $(M^n,g)$ is BTP, then $\nabla^b\eta=0$, hence $\nabla (\eta + \overline{\eta})=0$ and $g$ is Vaisman. Conversely, if $g$ is Vaisman, so we have $\nabla (\eta + \overline{\eta})=0$.
Then from the above calculation, we get
$$ \sum_{i,j} \{  \eta_{i,j} \,\varphi_j \otimes \varphi_i + \eta_{i,\overline{j}} \,\overline{\varphi}_j \otimes \varphi_i +  \overline{\eta_{i,\overline{j}}} \,\varphi_j \otimes \overline{\varphi}_i  + \overline{\eta_{i,j} } \ \overline{\varphi}_j \otimes \overline{\varphi}_i  \}  = 0.$$
Hence $\nabla^b \eta =0$ thus $\nabla^b T=0$ by (\ref{eq:4.11}). So we have proved that for a locally conformally K\"ahler manifold, the BTP condition is equivalent to the Vaisman condition.

For the second statement, let $g$ be a non-balanced BTP metric, so we may now apply the admissible frame $e$. If $g$ is Vaisman, the equality \eqref{eq:4.11} implies
$$ T^i_{in}=\frac{\lambda}{n-1}, \ \ \  \forall \ 1\leq i\leq n-1,$$
that is, $a_1=\cdots =a_{n-1}=\frac{\lambda}{n-1}$. Conversely, assume that $a_1=\cdots =a_{n-1}=\frac{\lambda}{n-1}$. We want to show that $g$ is locally conformally K\"ahler.
By Proposition \ref{prop1.7}, it is clear that $g$ is locally conformally balanced, i.e. $d(\eta+\overline{\eta})=0$, as all $a_i$ are real numbers.
Then it follows that, under the admissible frame $e$,
$$ d\overline{\varphi}_n = \partial \overline{\varphi}_n = \frac{1}{\lambda}  \partial \overline{\eta }
= \frac{\lambda}{n-1} \sum_{i=1}^{n-1} \varphi_i \wedge \overline{\varphi}_i, $$
which implies that $\omega = \sqrt{-1}( \frac{n-1}{\lambda} \partial \overline{\varphi}_n + \varphi_n \wedge \overline{\varphi}_n)$ and
$$ \partial \omega = - \sqrt{-1} \varphi_n \wedge \partial \overline{\varphi}_n
= - \sqrt{-1} \varphi_n \wedge \frac{\lambda}{n-1} ( \frac{1}{\sqrt{-1}} \omega - \varphi_n \wedge \overline{\varphi}_n)
= - \frac{\lambda}{n-1} \varphi_n \wedge \omega = -\frac{1}{n-1}\eta \wedge \omega.$$
That is, $d \omega = -\frac{1}{n-1}(\eta+\overline{\eta}) \wedge \omega$. This means that $g$ is locally conformally K\"ahler and we have
 completed the proof of Proposition \ref{prop1.8}.
\end{proof}

\begin{proof}[{\bf Proof of Proposition \ref{prop1.9}.}]
Suppose $(M^n,g)$ and $(M^n,e^{2u}g)$ are two BTP manifolds,
where $u$ is a smooth real-valued function on $M$ with $du\neq 0$ in an open dense subset $U\subseteq M$.
Let $e$ be a local unitary frame for $g$ with dual coframe $\varphi$.
Then $\tilde{e}=e^{-u}e$ and $\tilde{\varphi} = e^u\varphi$ are local unitary frame and dual coframe for $\tilde{g}=e^{2u}g$,
and we have, by \cite[the proof of Theorem 3]{YZZ},
\begin{eqnarray*}
e^u\tilde{T}^j_{ik} & = & T^j_{ik} + 2(u_i \delta_{kj} - u_k \delta_{ij}) \\
e^u \tilde{\eta}_k & = & \eta_k - 2(n-1)u_k \\
P_{ik} \ \ & = & 2u_i \varphi_k - \partial u \, \delta_{ik} - 2\overline{u_k} \, \overline{\varphi}_i + \overline{\partial}u \, \delta_{ik}
\end{eqnarray*}
where $P=\tilde{\theta}^b-\theta^b$ and $u_k=e_k(u)$. Take the covariant derivative
with respect to $\nabla^b_{\!e_{\ell}}$ of $g$ on both sides of the first equality above,
utilizing $\nabla^bT=0$ and $\tilde{\nabla}^b \tilde{T} =0$, so we get
\begin{equation}
u_{i,\ell}\delta_{kj} - u_{k,\ell} \delta_{ij} = 2u_kT^j_{i\ell} - 2u_i T^j_{k\ell} +2u_iu_{\ell} \delta_{kj} - 2u_k u_{\ell} \delta_{ij} - 2\delta_{j\ell} \sum_r u_r T^r_{ik}   \label{eq:4.12}
\end{equation}
for any indices $1\leq i,j,k,\ell \leq n$ under any local unitary frame $e$. As $du\neq 0$ holds on the open dense subset $U$, so within $U$ we may always choose our local unitary frame $e$ such that $\varphi_n= \frac{\partial u}{|\partial u|_g}$, which implies that $u_1=\cdots =u_{n-1}=0$ and $u_n=|\partial u|_g=f >0$. Assume $i<k$. If we take $j\notin \{ i,k\}$ in (\ref{eq:4.12}), we get
\begin{equation}
\delta_{kn} T^j_{i\ell} = \delta_{j\ell} T^n_{ik},   \qquad  i<k, \ j\neq i, k. \label{eq:4.13}
\end{equation}
Take $k=n$ and $j\neq \ell$, we get $T^j_{i\ell}=0$. Take $k<n$ and $j=\ell$, we get $T^n_{ik}=0$.
Hence under this particular unitary frame $e$, we have $T^j_{ik}=0$ whenever all three indices are distinct.
Now let us take $n=k$ and $j=\ell$ in (\ref{eq:4.13}), we get $T^j_{ij}=T^n_{in}$, for all $i,j<n$ and $i\neq j$.
It implies that $T^{\ell}_{\ell i}$ are equal for $\ell \neq i$ and $i<n$. Therefore
$T^{\ell}_{\ell i} = \frac{1}{n-1} \eta_i $ whenever $i<n$ and $\ell \neq i$, as $\sum_\ell T^{\ell}_{\ell i}=\eta_i$.

Then take $k=\ell =n$ and $i=j<n$ in (\ref{eq:4.12}), so we obtain
$$ -u_{n,n}= 2f T^i_{in} - 2f ^2,\qquad \ \forall \ i<n.$$
This shows that all $T^i_{in}$ for $i<n$ are equal, hence equal to $\frac{1}{n-1}\eta_n$.
Combining the results above, we conclude that under this particular frame $e$, it holds that
\begin{equation}
T^j_{ik} = \frac{1}{n-1} \big( \eta_k \delta_{ij} - \eta_i \delta_{kj}  \big),\qquad  \ \forall \ 1\leq i,j,k\leq n.   \label{eq:4.14}
\end{equation}
It amounts to $\tau= -\frac{1}{n-1} \eta \wedge \varphi$ by use of \eqref{eq:4.11}, or equivalently,
\[d \omega = -\frac{1}{n-1}(\eta + \overline{\eta}) \wedge \omega.\]
Note that the equation above holds on the open dense subset $U\subseteq M$ as this particular unitary frame exists around each point of $U$.
Hence it holds on the entire $M$. Then we calculate that
$$ \phi_i^j = \sum_r \overline{\eta}_r T^j_{ir} = \frac{1}{n-1}\left( \delta_{ij} \sum_r |\eta_r|^2 - \eta_i \overline{\eta}_j \right).$$
In particular, $\phi_i^j = \overline{\phi^i_j}$ for any $i,j$. By Proposition \ref{prop1.5}, we know that $\partial \eta =0$ and
$$ \overline{\partial }\eta + \partial \overline{\eta} = 2\sum_{i,j} \big( \phi_i^j - \overline{\phi^i_j}) \varphi_i\wedge \overline{\varphi}_j = 0,$$
which means $d(\eta + \overline{\eta}) =0$. Therefore $g$ is locally conformally K\"ahler. This completes the proof of Proposition \ref{prop1.9}
\end{proof}

We remark that locally there could be plenty of Vaisman metrics within a conformal class. We have the following

\begin{lemma}\label{LCK_VSM}
Let $(M^n,\tilde{g})$ be a K\"ahler manifold and $u$ a real-valued smooth function on $M$. Then the Hermitian metric $g=e^{2u}\tilde{g}$ is BTP (or equivalently, Vaisman) if and only if $u$ satisfy the following equations
\begin{equation} \label{eq:localVaisman}
 u_{ij} = 2u_iu_j + \sum_k u_k \tilde{\Gamma}^k_{ij}, \ \ \ \ u_{i\bar{j}}=2u_i\overline{u}_j - 2 |\partial u|^2_{\tilde{g}} \tilde{g}_{i\bar{j}}, \ \ \ \ \ \ \forall \ 1\leq i,j\leq n
 \end{equation}
where $(z_1, \ldots  z_n)$ is any local holomorphic coordinate system in $M$, $ \tilde{\Gamma}^k_{ij} = \sum_{\ell =1}^n \tilde{g}_{i\bar{\ell},j} \tilde{g}^{\bar{\ell}k}$ are the connection coefficients for $\tilde{g}$ and the subscripts after the coma stand for partial derivatives.
\end{lemma}

The proof is straight forward and analogous to the proofs of Propositions \ref{prop1.8} and \ref{prop1.9}, so we will omit it here.
As a particular example, let us consider the flat case when $M\subseteq {\mathbb C}^n$ is a domain and $\tilde{g}=g_0$ is the  Euclidean metric. In this case the equation (\ref{eq:localVaisman}) becomes
$$ u_{ij}=2u_iu_j, \ \ \ u_{i\bar{j}} = 2u_i \overline{u}_j - 2 \delta_{ij} \sum_k |u_k|^2, \ \ \ \ \ \ 1\leq i,j\leq n. $$
When $n=2$, the equation takes the simple form
\begin{equation} \label{eq:localVaisman2}
 u_{ij}=2u_iu_j, \ (1\leq i,j\leq 2), \ \ \ \ u_{1\bar{1}}=-2|u_2|^2, \ \ \ \  u_{1\bar{2}} = 2u_1\overline{u}_2, \ \ \ \ u_{2\bar{2}} = -2 |u_1|^2.
\end{equation}

Clearly, for any constants $a$, $b$, the function $u=-\frac{1}{2}\log (|z_1-a|^2+|z_2-b|^2)$, which is well-defined on ${\mathbb C}^2\setminus \{ (a,b)\}$,
satisfies equation (\ref{eq:localVaisman2}). So for domain  $M\subseteq {\mathbb C}^2$, the metrics
$$g = e^{2u}g_0 = \frac{1}{|z_1-a|^2+|z_2-b|^2}g_0$$
are all BTP (or equivalently, Vaisman), for any choice of $(a,b)\notin M$.

%\vspace{0.3cm}

Besides Vaisman manifolds, there are also examples of BTP manifolds amongst complex nilmanifolds, which form an important class of Hermitian manifolds that are often used to test and illustrate theory in non-K\"ahler geometry. See \cite{CFU}, \cite{GZZ}, \cite{LZ}, \cite{Sal}, \cite{Uga}, \cite{VYZ}, \cite{LS} as a sample of examples. Proposition \ref{prop1.10} is analogous to the main result of \cite{ZhaoZ19Nil} which describes Bismut K\"ahler-like (BKL) metrics within the class. Note that the description of the BTP metrics is actually simpler than the BKL metrics in the nilpotent case.

\begin{proof}[{\bf Proof of Proposition \ref{prop1.10}.}]
Let $(M^n,g)$ be a complex nilmanifold, namely, its universal cover is a nilpotent Lie group $G$ equipped with a left-invariant complex structure $J$ and a compatible left-invariant metric $g$. Assume that $J$ is nilpotent in the sense of \cite{CFGU} and $g$ is BTP. By Lemma 1, Lemma 2 and the discussion right after the proof of Lemma 2 in \cite{ZhaoZ19Nil}, we know that $G$ is of step at most 2 and $J$ is abelian (namely, $C=0$), and there exists an integer $1\leq r\leq n$ and a unitary coframe $\varphi$ such that
$$ d\varphi_i =0, \ \ \ \forall \ 1\leq i\leq r; \ \ \ \ \ \ d\varphi_{\alpha} = \sum_{i=1}^r Y_{\alpha i} \varphi_i \wedge \overline{\varphi}_i , \ \ \ \forall \ r<\alpha \leq n. $$
Conversely, for such a complex nilmanifold, it is easy to check that $\nabla^bT=0$, so the metric is BTP.
The characterization of balanced, BKL or Vaisman metric also follows by routine calculation.
This completes the proof of Proposition \ref{prop1.10}. \end{proof}

In particular, $G$  (when not abelian) is a 2-step nilpotent group and $J$ is abelian.
Let us take the $n=3$ case as an example. In this case, either $r=3$, where $G$ is abelian and $g$ is K\"ahler and flat,
or $r=2$ and there exists a unitary coframe $\varphi$ such that
$$ d\varphi_1=d\varphi_2=0, \ \ d\varphi_3 = a\varphi_1\wedge \overline{\varphi}_1 + b \varphi_2\wedge \overline{\varphi}_2, $$
where $a$, $b$ are constants. By a unitary change if necessary, we may assume that $a>0$. Under such a coframe, the metric $g$ is balanced if and only if $b=-a$,
$g$ is BKL if and only if $b\in \sqrt{-1}{\mathbb R}$ is pure imaginary,
while $g$ is Vaisman if and only if $b=a$.
For any other choice of $b$ values, $g$ is a non-balanced, non-BKL, non-Vaisman BTP metric. We will denote the following example by $N^3$,
\begin{equation}     \label{N_3}
d\varphi_1=d\varphi_2=0, \ \ \ \ \ d\varphi_3 = \varphi_1 \wedge \overline{\varphi}_1 - \varphi_2 \wedge \overline{\varphi}_2,
\end{equation}
In particular, starting in complex dimension $3$, there are (compact) BTP manifolds that are balanced (but non-K\"ahler).

Next let us prove Corollary \ref{cor1.11}. The proof is the same as the {\em BKL} case in \cite{ZhaoZ19Str}, and we include it here for the sake of completeness.

\begin{proof}[{\bf Proof of Corollary \ref{cor1.11}: }]
We know from the equality \eqref{eq:1.13} after Proposition \ref{prop1.5} that BTP metrics are always Gauduchon.
If $g$ is strongly Gauduchon, namely, if there exists a $(n,n-2)$-form $\Omega$ on $M^n$ such that $\partial \omega^{n-1}=\overline{\partial }\Omega$, then by the fact that $\partial \eta =0$ and $\partial \omega^{n-1}=-\eta \, \omega^{n-1}$, we get
\[ \int_M \eta \wedge \overline{\eta} \wedge \omega^{n-1} = \int_M \overline{\eta} \wedge \partial  \omega^{n-1} = \int_M \overline{\eta} \wedge \overline{\partial }\Omega = - \int_M \overline{\partial } (\overline{\eta} \wedge \Omega )= 0.\]
Hence $|\eta|^2 = 0$, contradicting to our assumption that $g$ is not balanced. So any compact, non-balanced BTP metric cannot be strongly Gauduchon. Similarly, if $\overline{\eta} = \overline{\partial }f$ for some smooth function $f$ on $M^n$, then
$$ \int_M \partial \overline{\eta} \wedge \omega^{n-1} = \int_M \partial \overline{\partial }f \wedge \omega^{n-1} =0$$
as $g$ is Gauduchon. Utilizing (\ref{eq:1.13}), we get $\eta =0$, which is again a contradiction. So $[\overline{\eta}]$ is not zero in $H^{0,1}(M)$.
Note that the above argument shows that the $(1,1)$-form $\partial \overline{\eta}$ is $\overline{\partial}$-closed and $\partial$-exact,
but cannot be $\partial \overline{\partial}$-exact. So the compact complex manifold $M$ does not satisfy the $\partial \overline{\partial}$-Lemma.
In particular, $M$ cannot admit any K\"ahler metric. This completes the proof of Corollary \ref{cor1.11}. \end{proof}

From the the last equation in (\ref{eq:1.12}), we know that on any BTP manifold it holds that
\begin{equation}
\phi + \phi^{\ast} = B -\mbox{Ric}(Q). \label{eq:4.8}
\end{equation}
Note that the right hand side of (\ref{eq:4.8}) is not necessarily non-negative, so the proof of \cite[the proof of Theorem 3]{ZhaoZ19Str} does not go through,
and that is why we cannot prove Conjecture \ref{conj1.12} here for non-balanced BTP manifolds.
If $B \geq \mbox{Ric}(Q)$, or equivalently, by the statement \eqref{eq:1.11} in Proposition \ref{prop1.5},
all the eigenvalues of the tensor $\phi$ have non-negative real parts, then Conjecture \ref{conj1.12} holds.

To finish this section, let us prove Proposition \ref{RicQ}.

\begin{proof}[\bf{Proof of Proposition \ref{RicQ}:}]
Let $(M^n,g)$ be a BTP manifold of complex dimension $n$ satisfying the condition $\mathrm{Ric}(Q)=0$, which is equivalent to $B=\phi + \phi^{\ast}$ by (\ref{eq:1.12}).  By the equality \eqref{eq:1.7} in Proposition \ref{prop1.5}, $Q=-P$ where
\[ P^{j\ell}_{\,ik} = \sum_r \big( T^r_{ik}\overline{T^r_{j\ell} } + T^j_{ir}\overline{T^k_{\ell r} } + T^{\ell}_{kr}\overline{T^i_{jr} }
- T^{\ell}_{ir}\overline{T^k_{jr} } - T^j_{kr}\overline{T^i_{\ell r} } \big).\]
Under any admissible frame $e$, we have $T^n_{ij}=0$ and $T^j_{in}=\delta_{ij}a_i$ for any $1\leq i,j\leq n$. So we get
\[P_{\,ik}^{jn}=(\overline{a}_j - \overline{a}_k - \overline{a}_i) T_{ik}^j=0, \ \ \ \ \forall \ 1\leq i,j,k\leq n,\]
where the last equality is due to (\ref{eq:1.10}). Therefore $P_{\,ik}^{j\ell}=0$ when any one of the four indices is $n$.
So when $n=3$,  the only possibly non-trivial components of $P$ is $P^{12}_{\,12}$.
But by the assumption $\mathrm{Ric}(Q)=0$, we have
\[0=\mathrm{Ric}(Q)_{1\bar{1}}=-\sum_r P_{\,1r}^{1r}=-P_{\,12}^{12},\]
therefore $P^{12}_{\,12}=0$ as well, so $P=0$ hence $Q=0$.

Now assume that $n=4$. Again under any admissible frame $e$, we have $P_{\,ik}^{j\ell}=0$ when any one of the four indices is $4$. So $P_{\,ik}^{j\ell}=0$ unless $i,j,k, \ell \leq 3$ and $i\neq k$, $j\neq \ell$. For any $1\leq i\neq j\leq 3$, let $k$ be the integer such that $\{ i,j,k\} = \{ 1,2,3\}$. Then we have
$$ 0 = \mathrm{Ric}(Q)_{i\bar{j}}=-\sum_r P_{\,ir}^{jr}=-P_{\,ik}^{jk}  .$$
Therefore whenever  the cardinality of the set $\{ i,j,k,\ell\} \subseteq  \{ 1,2,3\}$ is $3$ we will have $P_{\,ik}^{j\ell}=0$. That is, the only possibly non-zero components of $P$ are $P_{\,ik}^{ik}=S_{ik}$, $1\leq i\neq k\leq 3$. Clearly $S_{ik}=S_{ki}$. The assumption $\mathrm{Ric}(Q)=0$ now says that
$$ S_{12}+S_{13}= S_{12}+S_{23} = S_{13}+S_{23} = 0. $$
This implies that  $S_{12}=S_{13}=S_{23}=0$. So $P=0$ and $Q=0$.

When $n\geq 5$, however, the condition $\mathrm{Ric}(Q)=0$ no longer implies $Q=0$ for BTP manifolds in general. To illustrate this point, we will construct a $5$-dimensional BTP nilmanifold which satisfies $\mathrm{Ric}(Q)=0$ but with $Q \not\equiv 0$.
In light of Proposition \ref{prop1.10}, it is easy to check that $\mathrm{Ric}(Q)=0$ is equivalent to
\begin{equation}\label{RicQ_NL}
\sum_{\begin{subarray}{c} k=1 \\ k \neq i \end{subarray}}^r \sum_{\alpha = r+1}^n \{
Y_{\alpha i} \overline{Y_{\alpha k}} + \overline{Y_{\alpha i}} Y_{\alpha k} \} =0,\qquad \forall\ 1 \leq i \leq r.
\end{equation}
Let $r=4$ and $n=5$ to get a nilpotent Lie algebra endowed with a nilpotent complex structure by
\[\begin{cases}
d\varphi_i=0,   \qquad  \forall \ 1\leq i\leq 4, \\
d\varphi_{5} = \sum\limits_{i=1}^4Y_i\,\varphi_i \wedge \overline{\varphi}_i,
\end{cases}\]
where $Y_1=1+2\sqrt{-1}$, $Y_2=1-\sqrt{-1}$, $Y_3=1-\sqrt{-1}$ and $Y_4=\sqrt{-1}$.
It is obvious from the structure equation above the Lie algebra admits a rational structure
as the real and imaginary parts of $Y_1, Y_2, Y_3$ and $Y_4$ are rational,
so the example admits compact quotients by Malcev's Lemma thus determines a nilmanifold.
Clearly this example satisfies the equality \eqref{RicQ_NL} thus $\mathrm{Ric}(Q)=0$,
while $Y_1\overline{Y}_2 + \overline{Y_1}Y_2 \neq 0$ which implies that $Q \neq 0$.
\end{proof}

%\vspace{0.3cm}
\vspace{0.4cm}

\section{Non-balanced BTP threefolds}\label{NBBTP}

In this section we give a classification for non-balanced BTP threefolds. First let us recall the following definitions regarding \textit{Sasakian} manifolds, the \textit{standard Hermitian structure} and the \textit{modified ones} on the product of two Sasakian manifolds.

\begin{definition}\label{SSK}
Let $(N^{2m+1},g,\xi)$ be a \textit{Sasakian} manifold, that is, as defined in \cite{YZZ} for instance, an odd dimensional Riemannian manifold $(N,g)$ equipped with a Killing vector field $\xi$ of unit length, such that:
\begin{enumerate}
\item The tensor field $\frac{1}{c}\nabla \xi$ gives an integrable orthogonal complex structure $I$ on the distribution $H$, where $c>0$ is a constant and $H$ is the perpendicular complement of $\xi $ in the tangent bundle $TN$.
\item Denote by $\phi $ the $1$-form dual to $\xi$, namely, $\phi (X) = g(X, \xi)$ for any $X$, then $\phi \wedge (d\phi )^m$ is nowhere zero.
That is, $\phi$ gives a contact structure on $N$.
\end{enumerate}
\end{definition}

\begin{definition}[Belgun \cite{B12}]\label{SSKproduct}
Let $(N_1^{2n_1+1}, g_1, \xi_1)$ and $(N_2^{2n_2+1}, g_2, \xi_2)$ be two Sasakian manifolds. On the product Riemannian manifold $M=N_1\times N_2$, of even dimension $2n=2n_1+2n_2+2$,
for $\kappa \in \{\kappa \in \mathbb{C} \, \big| \, \mathrm{Im} (\kappa) > 0\}$, a natural almost complex structure $J_{\kappa}$ can be defined by
\begin{gather*}
J_{\kappa}\xi_1 = \mathrm{Re}(\kappa) \xi_1 + \mathrm{Im}(\kappa)\xi_2, \\
J_{\kappa}X_i = J_i X_i, \quad \forall \ X_i \in H_i , \ i=1,2.
\end{gather*}
One can also define an associated metric $g_{\kappa}$ by
\begin{gather*}
g_{\kappa}(\xi_1,\xi_1)= g_{\kappa}(J_{\kappa}\xi_1,J_{\kappa}\xi_1)=1, \quad g_{\kappa}(\xi_1,J_{\kappa}\xi_1)=0,\\
g_{\kappa}(\xi_i,X)=0,\quad \forall \ X \in H_1 \oplus H_2, \ i=1,2.\\
g_{\kappa}(X_i,Y_i)=g_i(X_i,Y_i), \quad g_{\kappa}(X_1,X_2)=0,\quad \forall \ X_i, Y_i \in H_i,\ i=1,2.
\end{gather*}
It is easy to verify that $g_{\kappa}$ is $J_{\kappa}$-Hermitian metric and $J_{\kappa}$ is integrable.
When $\kappa=y\sqrt{-1}$ for $y>0$, the Hermitian manifold defined above is the \textit{standard Hermitian structure on the product of two Sasakian manifolds} with $g_{\kappa}=g_1 \times \frac{g_2}{y^2}$, or \textit{standard Sasakian product}.
In general, we will call the modified Hermitian manifold $(N_1\times N_2, g_{\kappa}, J_{\kappa})$ \textit{twisted product of two Sasakian manifolds}, or \textit{twisted Sasakian product},
and denote it by $N_1 \times_{\kappa} N_2$.
\end{definition}
\begin{lemma}\label{twisted}
Let $(N_1,g_1,\xi_1),(N_2,g_2,\xi_2)$ be two $3$-dimensional Sasakian manifolds
and $\kappa = x + y \sqrt{-1}$ for $y >0$ and $x,y \in\mathbb{R}$.
Then there exists a natural unitary frame $(e_1,e_2,e_3)$, with the dual coframe $(\varphi_1,\varphi_2,\varphi_3)$,
on the twisted Sasakian product $N_1 \times_{\kappa} N_2$,
such that the Levi-Civita connection $\nabla$ under $e$ takes the form
\begin{equation}\label{eq_twSSK}
\begin{cases}
\nabla e_1 = c_1\sqrt{-1} \overline{\varphi}_1\, \xi_1 + \sigma_1\, e_1 \\
\nabla e_2 = c_2\sqrt{-1} \overline{\varphi}_2\, \xi_2 + \sigma_2\, e_2 \\
\nabla e_3 = \frac{c_1\sqrt{-1}}{\sqrt{2}}(\varphi_1\, e_1 - \overline{\varphi}_1 \, \overline{e}_1)
+ \frac{c_2 (1-x\sqrt{-1})}{y\sqrt{2}} (\varphi_2\, e_2 - \overline{\varphi}_2 \, \overline{e}_2),
\end{cases}
\end{equation}
where $e_3 = \frac{1}{\sqrt{2}}(\xi_1 - \sqrt{-1}(x\xi_1 + y \xi_2))$, $c_i$ is the defining constant of the Sasakian manifold $N_i$, and $\sigma_i$ is a $1$-form satisfying $\overline{\sigma}_i = - \sigma_i$, for $i=1,2$.
\end{lemma}
\begin{proof}
From Definition \ref{SSK}, for $i=1,2$, we may assume that $(\xi_i,e_i,\overline{e}_i)$ be a frame of the Sasakian manifold $(N_i, g_i, \xi_i)$ such that $e_i$ is a unitary frame of the distribution $\xi_i^{\bot}$ with respect to the orthogonal complex structure $\frac{1}{c_i} \nabla \xi_i$, which yields the following structure equation
\[ \nabla \! \begin{bmatrix} \xi_i \\ e_i \\ \overline{e}_i \end{bmatrix}=
\begin{bmatrix} 0 & c_i \sqrt{-1} \varphi_i & - c_i \sqrt{-1} \overline{\varphi}_i \\
c_i \sqrt{-1} \overline{\varphi}_i & \sigma'_i & 0  \\
-c_i \sqrt{-1} \varphi_i  & 0 & \overline{\sigma}'_i  \\ \end{bmatrix}\!\!
\begin{bmatrix} \xi_i \\ e_i \\ \overline{e}_i \\ \end{bmatrix}\!\!,\]
where $\sigma'_i$ is a $1$-form satisfying $\overline{\sigma}'_i = - \sigma'_i$ and the dual frame of $(\xi_i,e_i,\overline{e}_i)$ is denoted by $(\phi_i,\varphi_i,\overline{\varphi}_i)$. The dual version of the equation above is
\[ \nabla \! \begin{bmatrix} \phi_i \\ \varphi_i \\ \overline{\varphi}_i \end{bmatrix}=
\begin{bmatrix} 0 & - c_i \sqrt{-1} \overline{\varphi}_i & c_i \sqrt{-1} \varphi_i \\
-c_i \sqrt{-1} \varphi_i & \overline{\sigma}'_i & 0 \\
c_i \sqrt{-1} \overline{\varphi}_i & 0 & \sigma'_i  \\ \end{bmatrix}\!\!
\begin{bmatrix} \phi_i \\ \varphi_i \\ \overline{\varphi}_i \\ \end{bmatrix}\!\!.\]
As $\nabla$ is torsion free, it follows that
\[ d \!  \begin{bmatrix} \phi_i \\ \varphi_i \\ \overline{\varphi}_i \end{bmatrix}=
\begin{bmatrix} 0 & - c_i \sqrt{-1} \overline{\varphi}_i & c_i \sqrt{-1} \varphi_i \\
-c_i \sqrt{-1} \varphi_i & \overline{\sigma}'_i & 0 \\
c_i \sqrt{-1} \overline{\varphi}_i & 0 & \sigma'_i  \\ \end{bmatrix}\!\!
\begin{bmatrix} \phi_i \\ \varphi_i \\ \overline{\varphi}_i \\ \end{bmatrix}\!\!.\]

From Definition \ref{SSKproduct}, we may define
\begin{gather*}
e_3 = \frac{1}{\sqrt{2}}(\xi_1 - \sqrt{-1}J_{\kappa}\xi_1 ) = \frac{1}{\sqrt{2}}\big((1-x\sqrt{-1})\xi_1 - y\sqrt{-1} \xi_2 \big),\\
\varphi_3 = \frac{1}{\sqrt{2}}(\phi_1 - \frac{x}{y} \phi_2 + \frac{\sqrt{-1}}{y} \phi_2 ) =  \frac{1}{\sqrt{2}}(\phi_1 - \frac{x-\sqrt{-1}}{y} \phi_2 ).
\end{gather*}
It is easy to verify that $(e_1,e_2,e_3)$ is a unitary frame on the twisted Sasakian product $N_1 \times_{\kappa} N_2$ and  $(\varphi_1,\varphi_2,\varphi_3)$ is its dual coframe. Then it yields that
\[ \begin{cases}
d \varphi_1 = \frac{c_1}{\sqrt{2}}\big( (x + \sqrt{-1}) \varphi_3 - (x-\sqrt{-1}) \overline{\varphi}_3 \big)\varphi_1
+ \overline{\sigma}'_1 \varphi_1, \\
d \varphi_2 = \frac{c_2y}{\sqrt{2}}(\varphi_3 - \overline{\varphi}_3) \varphi_2 + \overline{\sigma}'_2 \varphi_2, \\
d \varphi_3 = \sqrt{2} c_1 \sqrt{-1} \varphi_1 \overline{\varphi}_1 - \frac{\sqrt{2}c_2 (1+x\sqrt{-1})}{y} \varphi_2 \overline{\varphi}_2.
\end{cases} \]
In matrix form we have $d \varphi = \overline{\theta}_1 \varphi + \theta_2 \overline{\varphi}$, where $\theta_1$ is skew Hermitian and $\theta_2$ is skew symmetric, given by
\begin{gather*}
\theta_1= \begin{bmatrix} \sigma_1 & 0 & \frac{c_1\sqrt{-1}}{\sqrt{2}} \overline{\varphi}_1  \\
                            0 & \sigma_2 & \frac{-c_2(1+x\sqrt{-1})}{y\sqrt{2}} \overline{\varphi}_2 \\
                            \frac{c_1\sqrt{-1}}{\sqrt{2}} \varphi_1 & \frac{c_2(1-x\sqrt{-1})}{y\sqrt{2}} \varphi_2 & 0 \\
                            \end{bmatrix}\!\!,\quad
\theta_2 = \begin{bmatrix} 0 & 0 & \frac{-c_1\sqrt{-1}}{\sqrt{2}} \varphi_1 \\
                           0 & 0 & \frac{c_2(1+x\sqrt{-1})}{y\sqrt{2}} \varphi_2 \\
                           \frac{c_1 \sqrt{-1}}{\sqrt{2}} \varphi_1 & \frac{-c_2(1+x\sqrt{-1})}{\sqrt{2}} \varphi_2 & 0 \\ \end{bmatrix}\!\!,\\
\sigma_1 = \sigma_1'-\frac{c_1 x}{\sqrt{2}}(\varphi_3-\overline{\varphi}_3),\qquad
\sigma_2 = \sigma_2' - \frac{c_2}{\sqrt{2}}\big((y-\frac{1-x\sqrt{-1}}{y})\varphi_3 - (y - \frac{1+x\sqrt{-1}}{y}) \overline{\varphi}_3 \big).
\end{gather*}
Note that $\nabla e = \theta_1 e + \overline{\theta}_2 \overline{e}$, $\xi_1 = \frac{e_3 + \overline{e}_3}{\sqrt{2}}$ and $\sqrt{-1}\xi_2 =  \frac{-(1+x\sqrt{-1})}{y\sqrt{2}} e_3 +  \frac{(1-x\sqrt{-1})}{y\sqrt{2}} \overline{e}_3$. Then the equation \eqref{eq_twSSK} in the lemma is established.
\end{proof}

\begin{proof}[{\bf Proof of Theorem \ref{3DNBBTP}:}]
Let $(M^3,g)$ be a non-balanced BTP manifold. By Proposition \ref{prop1.7}, there always exist admissible frames.
Let $e$ be an admissible frame. Then the Bismut connection matrix $\theta^b$ under $e$ satisfies $\theta^b_{3\ast} = \theta^b_{\ast 3}=0$.
The eigenvalues of $\phi$ are $\{ \lambda a_1, \lambda a_2, 0\}$.
For simplicity of the writing below, let us write $2a=a_1$ and $2b=a_2$. We have $a+b=\frac{1}{2}\lambda$.
As the torsion is degenerate by Proposition \ref{BTP_LP} when $n=3$,
the only possibly non-zero Chern torsion components under $e$ are
\[T_{13}^1=2a,\quad T^2_{23}=2b.\]
Let us start with Case \eqref{twSSK} in Theorem \ref{3DNBBTP}, which means $a\bar{b}+\bar{a}b\neq 0$, $a\bar{b}-\bar{a}b\neq 0$. Clearly, in this case  $a,b\neq 0$ and $a \neq b$. So the connection matrix $\theta^b$  under $e$ is diagonal, denoted by
\[\theta^b = \begin{bmatrix} \psi_1 & 0 & 0 \\
                               0    & \psi_2 & 0 \\
                               0  & 0 & 0 \\ \end{bmatrix}\!\!,\]
where $\overline{\psi}_1 = - \psi_1$ and $\overline{\psi}_2 = - \psi_2$.
Since the connection matrix of the Levi-Civita connection $\nabla$ is given by $\nabla e = \theta_1 e + \overline{\theta}_2 \overline{e}$, where $\theta_1 = \theta^b - \frac{1}{2}\gamma$ and $\gamma,\theta_2$ are determined by the Chern torsion $T$, we have
\begin{equation}\label{generic}
\begin{cases}
\nabla e_1 =  \overline{\varphi}_1 \, (a \overline{e}_3 - \overline{a} e_3 ) + \sigma_1 \, e_1 \\
\nabla e_2 = \overline{\varphi}_2\,(b \overline{e}_3 - \overline{b} \overline{e}_3) + \sigma_2 \, e_2 \\
\nabla e_3 = a (\varphi_1\, e_1 - \overline{\varphi}_1 \, \overline{e}_1)
+ b (\varphi_2\, e_2 - \overline{\varphi}_2 \, \overline{e}_2),
\end{cases}
\end{equation}
where $\varphi$ is the dual coframe of $e$, and $\sigma_1 = \psi_1 - a \varphi_3 + \overline{a} \overline{\varphi}_3$, $\sigma_2 = \psi_2 - b \varphi_3 + \overline{b} \overline{\varphi}_3$. Then we may define two global real vector fields of unit length by
\[ \xi_1 = \frac{-\sqrt{-1}}{\sqrt{2}|a|}(a \overline{e}_3 - \overline{a} e_3 ),\quad
\xi_2 = \frac{- \sqrt{-1}}{\sqrt{2}|b|}(b \overline{e}_3 - \overline{b} e_3 ),\]
where it is easy to verify from the Case \eqref{twSSK} assumption on $a$, $b$ that $\mbox{span}_{\mathbb{C}} \{e_3,\overline{e}_3\} = \mbox{span}_{\mathbb{C}} \{\xi_1,\xi_2\}$, and it holds that
\[e_3 = \frac{\sqrt{2}|a|b \sqrt{-1}}{a\overline{b}-\overline{a}b} \xi_1
- \frac{\sqrt{2} |b| a \sqrt{-1}}{a\overline{b} - \overline{a} b} \xi_2,\quad
J \xi_1 = \frac{-(a\overline{b}+\overline{a}b)\sqrt{-1}}{a\overline{b}-\overline{a}b} \xi_1
+ \frac{2|a||b|\sqrt{-1}}{a\overline{b}-\overline{a}b} \xi_2. \]
Then it is obvious that, for $i=1,2$, the distributions $E_i$ generated by $\{\mathrm{Re}(e_i),\mathrm{Im}(e_i),\xi_i\}$ are globally defined on $M$ and they are both foliations, since the equation \eqref{generic} implies $[e_i,\overline{e}_i] \in E_i$ and $[e_i,\xi_i] \in E_i$.
So the universal covering space $\widetilde{M}$ of $M$ is a product manifold $N_1 \times N_2$, where $N_i$ is the integral manifold of the foliation $E_i$.
On the other hand, since neither $E_1$ nor $E_2$ are necessarily $\nabla$-parallel, $\widetilde{M}$ is not a metric product in general.
To simplify the writing, let us denote by
\begin{gather*}
x=\frac{-(a\overline{b}+\overline{a}b)\sqrt{-1}}{a\overline{b}-\overline{a}b},\quad
y=\frac{2|a||b|\sqrt{-1}}{a\overline{b}-\overline{a}b}, \quad \kappa = x + y\sqrt{-1}, \\
e'_3 = \frac{1}{\sqrt{2}}(\xi_1 - \sqrt{-1}J\xi_1)= \frac{1}{\sqrt{2}}\big((1-x\sqrt{-1})\xi_1 - y\sqrt{-1} \xi_2 \big).
\end{gather*}
It is easy to see that $x,y$ are real constants and $e'_3=\frac{\sqrt{-1} \overline{a}}{|a|}e_3$.
Then the first two equations of \eqref{generic} enable us to establish two Sasakian structure on $N_1$ and $N_2$ by the following two structure equations respectively,
\[ \nabla \! \begin{bmatrix} \xi_1 \\ e_1 \\ \overline{e}_1 \end{bmatrix}=
\begin{bmatrix} 0 & \sqrt{2}|a| \sqrt{-1} \varphi_1 & - \sqrt{2}|a| \sqrt{-1} \overline{\varphi}_1 \\
\sqrt{2}|a| \sqrt{-1} \overline{\varphi}_1 & \sigma'_1 & 0  \\
-\sqrt{2}|a| \sqrt{-1} \varphi_1  & 0 & \overline{\sigma}'_1  \\ \end{bmatrix}\!\!
\begin{bmatrix} \xi_1 \\ e_1 \\ \overline{e}_1 \\ \end{bmatrix}\!\!,\]
\[ \nabla \! \begin{bmatrix} \xi_2 \\ e_2 \\ \overline{e}_2 \end{bmatrix}=
\begin{bmatrix} 0 & \sqrt{2}|b| \sqrt{-1} \varphi_2 & - \sqrt{2}|b| \sqrt{-1} \overline{\varphi}_2 \\
\sqrt{2}|b| \sqrt{-1} \overline{\varphi}_2 & \sigma'_2 & 0  \\
-\sqrt{2}|b| \sqrt{-1} \varphi_2  & 0 & \overline{\sigma}'_2  \\ \end{bmatrix}\!\!
\begin{bmatrix} \xi_2 \\ e_2 \\ \overline{e}_2 \\ \end{bmatrix}\!\!,\]
where the dual coframe of $(\xi_i,e_i,\overline{e}_i)$ is denoted by $(\phi_i,\varphi_i,\overline{\varphi}_i)$ for $i=1,2$, and we may also define
\begin{gather*}
\varphi'_3 = \frac{1}{\sqrt{2}}(\phi_1 - \frac{x}{y} \phi_2 + \frac{\sqrt{-1}}{y} \phi_2 ) =  \frac{1}{\sqrt{2}}(\phi_1 - \frac{x-\sqrt{-1}}{y} \phi_2 ),\\
\sigma_1'= \sigma_1 + x|a|(\varphi'_3 - \overline{\varphi'_3}),\quad \sigma'_2 = \sigma_2 + |b| \big( (y- \frac{1-x\sqrt{-1}}{y}) \varphi'_3 - (y-\frac{1+x\sqrt{-1}}{y}) \overline{\varphi'_3} \big).
\end{gather*}
Here $\sigma_1, \sigma_2$ are coefficient forms from the equation \eqref{generic}. If $y<0$, we may replace $\xi_2$ and $e_2$ by $-\xi_2$ and $\overline{e}_2$ respectively to obtain a Sasakian 3-manifold opposite to $N_2$ above
so that $y$ changes the sign. Hence it follows from Lemma \ref{twisted} that the structure equation of $N_1 \times_{\kappa} N_2$ is exactly the one \eqref{generic},
after we note that $c_1 = \sqrt{2}|a|, c_2 = \sqrt{2}|b|, e'_3=\frac{\sqrt{-1} \overline{a}}{|a|}e_3$ here. Therefore, $\widetilde{M}$ is holomorphically isomorphic to $N_1 \times_{\kappa} N_2$ and we have shown the Case \eqref{twSSK}.

Next we consider Case \eqref{BKL}. In this case, either one of $a$, $b$ is zero, for instance $a=0$ and $b=\frac{\lambda}{2}$, or $a=\frac{\lambda}{4}(1+\rho)$, $b=\frac{\lambda}{4}(1-\rho)$ for some $\rho$ satisfying $|\rho|=1$ and $\mathrm{Im}(\rho)>0$.
It is easy to verify that $P_{ik}^{j \ell}=0$ under the admissible frame $e$ in this case, thus $(M^3,g)$ is BKL by Proposition \ref{prop3.6}.
The universal covering space $\widetilde{M}$ is holomorphically isomorphic to either the product of a K\"ahler curve and a BKL surface, or the product of two Sasakian $3$-manifolds as shown in \cite[Theorem 7]{YZZ}.

Finally in Case \eqref{gV} which means $a$ and $b$ satisfies $a\bar{b}+\bar{a}b\neq 0$ and $a\bar{b}-\bar{a}b = 0$.
In this case both $a$ and $b$ are non-zero, with $a\bar{b}$ real. Since $a+b=\frac{\lambda}{2}>0$, both $a$, $b$ are real (and non-zero).
Conversely, when both $a$, $b$ are real and non-zero, we certainly have $a\bar{b}+\bar{a}b\neq 0$ and $a\bar{b}-\bar{a}b = 0$.
In this case, we know that $(M^3,g)$ is locally conformally balanced from Proposition \ref{prop1.7}.
Consider the two global real vector fields of unit length defined by
\[ \xi_1 = \frac{-\sqrt{-1}}{\sqrt{2}|a|}(a \overline{e}_3 - \overline{a} e_3 ),\quad
\xi_2 = J \xi_1 =\frac{- a}{\sqrt{2}|a|}( \overline{e}_3 + e_3 ).\]
Clearly $\mbox{span}_{\mathbb{C}} \{e_3,\overline{e}_3\} = \mbox{span}_{\mathbb{C}} \{\xi_1,\xi_2\}$ and $\nabla \xi_2=0$. Consider the distributions defined by
$$ F_1= \mbox{span}_{\mathbb{R}}\{\mathrm{Re}(e_1),\mathrm{Im}(e_1), \mathrm{Re}(e_2),\mathrm{Im}(e_2),\xi_1\}, \ \ \ \ F_2= \mbox{span}_{\mathbb{R}}\{ \xi_2 \}. $$
Then both $F_1$ and $F_2$ are globally defined on $M$ and are $\nabla$-parallel. Hence the universal covering space $\widetilde{M}$ is holomorphically isomorphic to $N \times \mathbb{R}$, where $N$ is generated by $F_1$.
Here we may call $N$ a {\em generalized $5$-Sasakian manifold.}
Note that in the special case  when $a=b=\frac{\lambda}{4}$, we have $d \omega = - \frac{1}{2}(\eta + \overline{\eta}) \omega$ and $(M^3,g)$ is a Vaisman threefold by Proposition \ref{prop1.8}.
We have completed the proof of Theorem \ref{3DNBBTP}.
\end{proof}

\begin{proof}[{\bf Proof of Corollary \ref{cor1.16}.}]
Let $(M^3,g)$ be a non-balanced BTP threefold. As the torsion is degenerate by Proposition \ref{BTP_LP} when $n=3$,
under an admissible frame $e$, we have $T^1_{13}=a_1$, $T^2_{23}=a_2$ which are the only possibly non-zero Chern torsion components.
Since $\nabla^be_3=0$, the Bismut connection matrix $\theta^b$ under $e$ satisfies $\theta^b_{3i}=\theta^b_{i3}=0$ for any $i$. So $\nabla^bT=0$ implies
$$ 0 = dT^2_{13} = \sum_r \big( T^2_{r3} \theta^b_{1r} + T^2_{1r} \theta^b_{3r} - T^r_{13} \theta^b_{r2} \big) = (a_2-a_1)\,\theta^b_{12}.$$
When $a_1\neq a_2$, we have $\theta^b_{12}=0$ hence $\theta^b$ is diagonal, and the Bismut curvature matrix $\Theta^b$ is also diagonal.
If we write the distribution generated by $e_i$ as $L_i$, for $1 \leq i \leq 3$, it is clear that $L_1, L_2, L_3$ are all globally defined (complex) line bundles, with the first two probably not holomorphic while the third holomorphic, as they are served as the eigenspaces of $\phi$ with respect to the eigenvalues $\lambda a_1, \lambda a_2,0$ respectively, which induces the decomposition $T_{M^3}=L_1 \oplus L_2 \oplus L_3$ of the holomorphic tangent bundle $T_{M^3}$. So in this case the global Bismut holonomy group $\mbox{Hol}(\nabla^b)$ is actually contained in $U(1)\times U(1)\times 1$.
When $a_1=a_2$, each is necessarily equal to $\frac{\lambda}{2}$ as $a_1+a_2=\lambda$. Proposition \ref{prop1.8} implies that $(M^3,g)$ is Vaisman. This completes the proof of Corollary \ref{cor1.16}.
\end{proof}

\begin{proof}[{\bf Proof of Corollary \ref{cor1.17}.}]
Let $(M^3,g)$ be a non-balanced BTP threefold. Again under an admissible frame $e$,
we have $T^1_{13}=a_1$, $T^2_{23}=a_2$ as the only possibly non-zero Chern torsion components
by the degenerate torsion condition from Proposition \ref{BTP_LP}.
Write $c=a_1\overline{a}_2+\overline{a}_1a_2$. By Proposition \ref{plcld}, we have
$$ \partial \overline{\partial} \omega = \frac{\sqrt{-1}}{4} \sum_{i,j,k,\ell} P^{j\ell }_{\,ik} \varphi_i\varphi_k \overline{\varphi}_j\overline{\varphi}_{\ell} , \ \ \ \ \ P^{j\ell}_{\,ik} = \sum_r \big( T^r_{ik}\overline{T^r_{j\ell} }   + T^j_{ir}\overline{T^k_{\ell r} } + T^{\ell}_{kr}\overline{T^i_{jr} } - T^{\ell}_{ir}\overline{T^k_{jr} } - T^j_{kr}\overline{T^i_{\ell r} } \big) .$$
Under our admissible frame, it is easy to see that $P^{j\ell}_{\,ik}=0$ whenever any of the indices is $3$, and the only non-trivial components of $P$ is $P^{12}_{\,12}=c$.
Thus $  \partial \overline{\partial} \omega = \frac{\sqrt{-1}}{4}c  \,\varphi_1\varphi_2 \overline{\varphi}_1\overline{\varphi}_2$.
Now assume that $g_0$ is a pluriclosed Hermitian metric on $M^3$, namely, its K\"ahler form $\omega_0$ satisfies $\partial \overline{\partial} \omega _0=0$. Then we have
$$ 0 = \int_M \omega \wedge \partial \overline{\partial} \omega_0=  \int_M \partial \overline{\partial} \omega \wedge \omega_0= \frac{\sqrt{-1}}{4}c \int_M \varphi_1\varphi_2 \overline{\varphi}_1\overline{\varphi}_2 \wedge \omega_0 , $$
as $\int_M \varphi_1\varphi_2 \overline{\varphi}_1\overline{\varphi}_2 \wedge \omega_0 \neq 0$, which implies $c=0$ and thus $g$ is BKL.
The corollary is thus proved.
\end{proof}

%\vspace{0.3cm}

\vspace{0.3cm}

\section{Statements and declarations}

\noindent {\bf Statement on Conflict of Interest.}

All authors declare that they have no conflict of interest.

\vspace{0.3cm}

\noindent {\bf Data Availability Statement.}

Data sharing not applicable to this article as no datasets were generated or analysed during the current study.

%\vspace{0.3cm}
%
%
%\vspace{0.3cm}

\vspace{1cm}

\textbf{Acknowledgement}: We are grateful to the anonymous referees for their careful examination
and useful suggestion on the improvement of this paper.

\end{document}